\documentclass[review]{article}

\usepackage{lipsum}
\usepackage{amsfonts}
\usepackage{graphicx}
\usepackage{epstopdf}
\usepackage{algorithmic}
\ifpdf
  \DeclareGraphicsExtensions{.eps,.pdf,.png,.jpg}
\else
  \DeclareGraphicsExtensions{.eps}
\fi

\usepackage{amsopn}
\DeclareMathOperator{\diag}{diag}

\usepackage{pgfplotstable} 

\usepackage{multirow}
\usepackage{algorithm}
\usepackage{algorithmic}

\usepackage{subcaption}

\usepackage{xcolor}
\usepackage{booktabs}
\usepackage{comment}
\usepackage{amsmath}
\usepackage{mathrsfs}
\usepackage{float}
\usepackage{epstopdf}
\usepackage{multirow}
\usepackage{rotating}
\usepackage{amssymb}

\usepackage{amsthm}
\newtheorem{theorem}{Theorem}
\newtheorem{lemma}{Lemma}
\newtheorem{proposition}{Proposition}
\newtheorem{corollary}{Corollary}
\newtheorem{definition}{Definition}
\newtheorem{cnj}{Conjecture}

\newtheorem{remark}{Remark}

\newtheorem{claim}{Claim}

\DeclareMathOperator{\extr}{extr}

\usepackage{rotating}
\usepackage{pdflscape}

\usepackage{tikz,pgfplots}
\usetikzlibrary{matrix}
\usepgfplotslibrary{groupplots}
\pgfplotsset{compat=newest}
\usetikzlibrary{calc,positioning,shapes,fit,arrows,shadows}
\tikzstyle{line} = [ draw, -latex']  
\usetikzlibrary{shapes,arrows}
\usepgfplotslibrary{units}
\usepackage{url}
\usepackage{color}

\graphicspath{{figures/}} % Graphics will be here

\usepackage{color}

\usepackage[shortlabels]{enumitem}

\allowdisplaybreaks

\def \OO{ {\mathcal{O}}}

\def \SS{ {\mathcal{S}}}

\def \CP{ {\mathbb{CP}}}

\def \rr{ {\mathbb{R}}}

 \def \ZZ{ {\mathbb{Z}}}

\def \conv{\textup{conv}}

\def \RR{\mathbb{R}}

\newcommand{\cS}{\mathcal{S}}
\newcommand{\cG}{\mathcal{G}}

\newcommand{\cT}{\mathcal{T}}
\newcommand{\cC}{\mathcal{C}}
\newcommand{\cE}{\mathcal{E}}
\newcommand{\cD}{\mathcal{D}}

\usepgfplotslibrary{fillbetween}

\title{Convexification of a Separable Function over a Polyhedral Ground Set
}

\author{Santanu S. Dey\thanks{H. Milton Stewart School of Industrial and Systems Engineering, Georgia Institute of Technology
  (\texttt{santanu.dey@isye.gatech.edu}).} 
  \and
  Burak Kocuk\thanks{Industrial Engineering Program,  Sabanc{\i} University
  (\texttt{burakkocuk@sabanciuniv.edu}).}
  }

\begin{document}

\maketitle

% REQUIRED
\begin{abstract}
In this paper, we study the set $\mathcal{S}^\kappa = \{ (x,y)\in\mathcal{G}\times\mathbb{R}^n : y_j = x_j^\kappa , j=1,\dots,n\}$, where $\kappa > 1$ and the \textit{ground set} $\mathcal{G}$ is a nonempty polytope contained in $[0,1]^n$. This nonconvex set is closely related to separable standard quadratic programming and appears as a substructure in potential-based network flow problems from gas and water networks. Our aim is to obtain the convex hull of $\mathcal{S}^\kappa$ or its tight outer-approximation for the special case when the ground set $\mathcal{G}$ is the standard simplex. We propose power cone, second-order cone and semidefinite programming relaxations for this purpose, which are further strengthened by the Reformulation-Linearization Technique and the Reformulation-Perspectification Technique. For $\kappa=2$, we obtain the convex hull of $\mathcal{S}^\kappa$ in the low-dimensional setting. For general $\kappa$, we give approximation guarantees for the power cone representable relaxation, the weakest relaxation we consider. We prove that this weakest relaxation is tight with probability one as $n\to\infty$ when a uniformly generated linear objective is optimized over it. Finally, we provide the results of our extensive computational experiments comparing the empirical strength of several conic programming relaxations that we propose.
\end{abstract}

% REQUIRED

\section{Introduction}

Consider the    set 
\begin{equation*}\label{eq:generic}
        \cS^\kappa = \{ (x,y)\in\cG\times\RR^n : y_j = x_j^\kappa , j=1,\dots,n\},
\end{equation*}
where   $\kappa > 1$ and the \textit{ground set}
 $\cG$ is the following nonempty polytope
\begin{equation*}\label{eq:ground}
        \cG = \{ x \in [0,1]^n :   \ Ax = b, \ Cx \le d  \}.
\end{equation*}
 Here, $A\in \RR^{m\times n}$, $b\in \RR^m$, $C\in \RR^{k\times n}$ and $d\in \RR^k$. Notice that this set is closely related to the convexification of the separable function $\sum_{j=1}^n (\alpha_j x_j + \beta_j x_j^\kappa)$ over a polyhedral ground set $\cG$ and the optimization problem $\min_{(x,y)\in\cS^\kappa}\{\sum_{j=1}^n (\alpha_j x_j + \beta_j y_j)\}$. 
 This substructure appears in many applications and  we provide two motivating examples to study the set $\cS^\kappa$ below.

 The first example arises from a special case of the well-studied standard quadratic programming problem~\cite{bomze2002solving,bomze2018complexity,Liuzzi02012019,gondzio2021global,gokmen2022standard,judice2024two}, called the \textit{separable} standard quadratic programming~\cite{bomze2012separable}, which is an optimization problem of the form 
\begin{equation}\label{eq:separableStQuadProg}
      \min_{x\in\Delta^n} \left \{ \sum_{j=1}^n (\alpha_j x_j + \beta_jx_j^2) \right\},
 \end{equation}
 where $\alpha\in\mathbb{R}^n$,  $\beta\in\mathbb{R}^n$ and  the set $ \Delta^n:=\left\{x\in\RR^n_+ : \sum_{j=1}^n x_j = 1 \right\}$ is the standard simplex in $\mathbb{R}^n$. In our notation, problem~\eqref{eq:separableStQuadProg} is equivalent to 
 \( \min_{(x,y)\in \cS^2 } \{\alpha^Tx+\beta^Ty\} \), 
 with the ground set $\cG=\Delta^n$. Therefore, the set we study is a generalization of the extended formulation of the feasible region of an important problem.

 The second example arises from potential-based flow networks (e.g., gas and water networks~\cite{Gross2019,Rios2015,li2024reformulation,d2015mathematical,okumusoglu2025global,li2025optimizing,borner2025valid}). 
Let us consider a ``node-based" substructure, where we have a   node set $\{0,1,\dots,n\}$ and an edge set $\{(1,0),\dots,(n,0)\}$; see Figure~\ref{fig:network} for an illustration with $n=4$ (we also refer the reader to~\cite{dey2022node} for a similar substructure studied in the context of power systems).  
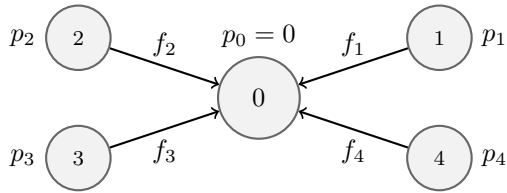
\begin{figure}[h]
\centering
\begin{tikzpicture}[scale=.80,
    node distance=0.5cm,
    source/.style={circle, draw=black!60, fill=black!5, thick, minimum size=4mm, text width=7.5mm, 
                align=center,},
    neigh/.style={circle, draw=black!60, fill=black!5, thick, 
                minimum size=4mm, text width=7.5mm, 
                align=center, font=\footnotesize, inner sep=1pt}
]

% Source  
\node[source, label=above:{$p_0=0$}] (source) at (0,0) {0} ; 

% Neighboring Nodes  
\node[neigh, label=right:{$p_1$}] (neigh1) at (3,1) {1};
\node[neigh, label=left:{$p_2$}] (neigh2) at (-3,1) {2};
\node[neigh, label=left:{$p_3$}] (neigh3) at (-3,-1) {3};
\node[neigh, label=right:{$p_4$}] (neigh4) at (3,-1) {4}; 
 
% Arcs
\draw[<-, thick] (source) -- (neigh1) node [pos=0.5, above] {$f_1$};
\draw[<-, thick] (source) -- (neigh2) node [pos=0.5, above] {$f_2$};
\draw[<-, thick] (source) -- (neigh3) node [pos=0.5, below] {$f_3$};
\draw[<-, thick] (source) -- (neigh4) node [pos=0.5, below] {$f_4$};

\end{tikzpicture}
    \caption{An example node-based structure with $n=4$ adjacent edges to node 0.}
\label{fig:network}
\end{figure}

\noindent
 We will think of node $0$ as a demand node with a positive demand $\delta$ and the others as supply nodes.  Each node $j$ has a nonnegative \textit{potential}, denoted by $p_j$, $0=1,\dots,n$, and we will assume that   the demand node 0 has zero potential, i.e., $p_0=0$. Each edge has a nonnegative \textit{flow}, denoted by $f_j$, which is upper bounded by $\bar f_j$, $j=1,\dots,n$.
%We will  assume that the demand node has zero potential, i.e., $y_0=0$. 
The relationship between the potential differences and the flow value is governed by the law of physics and is of the form $p_j-p_0 = \chi_j f_j^\kappa  $, where $\chi_j$ is a positive number representing a certain physical characteristic of edge~$j$. In addition, the relation $\sum_{j=1}^n f_j=\delta$ models the demand constraint.  By defining $x_j= \frac{f_j}{\bar f_j} $ and $y_j= \frac{p_j}{\chi_j \bar f_j^\kappa} $, we can model this situation in the form of $\cS^\kappa$, where the ground set is $\cG=\{ x\in[0,1]^n : \sum_{j=1}^n \bar f_j x_j = \delta\}$. Note that in the special case where we have $\chi_j=\bar f_j=\delta=1$ for $j=1,\dots,n$, the ground set is again the standard simplex, that is, $\cG=\Delta^n$. 
As opposed to our illustration above,  the directions of  flows on the edges are typically not  fixed in general. However, the analysis of the set $\cS^\kappa$ is still useful as one can deal with a more general case in which the flow directions are variables via a disjunctive formulation that involves several sets of the form $\cS^\kappa$. 
%\[
%  \left\{ (x, y)\in [\underline h, \overline h] \times [\underline f, \overline f] : h_j - h_0 = \beta_j \sign(f_j) f_j^\kappa, j=1,\dots,n ,   \ \sum_{j=1}^n f_j = d \right\} . 
%\]
%It is easy to see that in an extreme point of this set, at least one of the $h_j$ variables must be at one of its variable bounds. 
%
%Below, we first analyze a simple case where $h_0$ is fixed to 0, all the other variables are assumed to be  bounded between 0 and 1,  the demand is equal to $d=1$ and $\kappa=2$.

 In this study, our aim is to find $\conv(\cS^\kappa)$ or a close outer-approximation of $\conv(\cS^\kappa)$ in the space of $x$ and $y$ variables or in an extended space. Although the convexification of functions is an active research area (see, e.g., \cite{sherali1997convex,Rikun1997,ryoo2001analysis,tawarmalani2002convexification,benson2004concave,meyer2005convex,belotti2010valid,bao2015global,tuy2016convex,
adams2019error,he2024convexification,xu2025gaining}), the specific substructure we are interested in is less explored. Arguably,  references \cite{burer2009copositive}, \cite{bomze2012separable} and \cite{gokmen2022standard} are  most closely related  to our study although they only consider the case where the exponent is $\kappa=2$. For instance, the results in \cite{burer2009copositive} imply that $\conv(\cS^2)$ can be obtained as a completely positive cone representable set although this representation is not tractable in general. It is proven in \cite{bomze2012separable} that the optimizing a linear function over $\cS^2$ can be done is polynomial-time although an explicit convex hull description is not provided. In \cite{gokmen2022standard}, it is shown that optimizing a linear function over $\cS^2$ can be performed by solving a doubly nonnegative relaxation when $\beta$ vector is either nonnegative or nonpositive.  

In this paper, we provide several conic representable outer-approximations for the set $\cS^\kappa$ or equivalently conic programming relaxations for the optimization problem $\min_{(x,y)\in\cS^\kappa}\{\sum_{j=1}^n (\alpha_j x_j + \beta_j y_j)\}$. In particular, we propose a  power cone relaxation (hereafter abbreviated as the \textbf{P} relaxation) obtained as a \textit{single row relaxation}, that is, the nonconvex relation $y_j=x_j^\kappa$ over $x_j\in[0,1]$ is simply relaxed as $x_j^\kappa\le y_j\le x_j$. We also construct stronger relaxations using reformulation-linearization technique (RLT), second-order cone programming and semidefinite programming. Then, we specialize our analysis to the important special case when the ground set is the standard simplex. For this setting, we provide convex hull results in low-dimension for $\kappa=2$, and approximation guarantees and a probabilistic tightness analysis for general $\kappa$. We further improve our relaxations using the reformulation-perspectification technique (RPT). Finally, we run an extensive set of computational experiments focusing on the case  when the ground set is the standard simplex and provide empirical evidence of the comparative strength of the nine relaxations we propose.

%Surprisingly, this seemingly weak relaxation has several interesting properties as we demonstrate in the sequel. We further improve this relaxation using techniques

%Reformulation-Perspectification Technique (RPT)    {\color{red}
%(reference?) \cite{zhen2021extension,bertsimascone,dey2025second}
%[are these the best references? they are unpublished], we derive power cone representable valid inequalities when the ground set is the standard simplex}

Below, we list our main contributions and key results from each section:
\begin{itemize}
    \item In Section~\ref{sec:prelim}, we prove that optimizing a linear function over $\cS^\kappa$ is NP-hard in Proposition~\ref{prop:complexityNP-Hard} and propose several conic programming relaxations.
    \item In Section~\ref{sec:main}, we analyze the case where the ground set is the standard simplex. 
    \begin{itemize}
        \item For the special case $\kappa=2$, we provide two convex hull results in low-dimensional setting (Theorem~\ref{thm:exactness-kappa=2-stSimplexn=2} for $n=2$ and Corollary~\ref{cor:CPP} for $n=3$) and give some sufficient conditions under which a doubly nonnegative relaxation of the problem 
    $\min_{(x,y)\in\cS^\kappa}\{\sum_{j=1}^n (\alpha_j x_j + \beta_j y_j)\}$ is exact (Theorem~\ref{thm:exactness-kappa=2-specialCases}) for general $n$.
         \item For general $\kappa$, we provide both distance-based (Propositions~\ref{prop:dbound-stSimplex-ub} and \ref{prop:dbound-stSimplex-lb})  and objective-function based (Proposition~\ref{prop:Obound-stSimplex-kappaGeneral}) approximation results for the \textbf{P} relaxation. Interestingly, this weakest relaxation we consider  is tight with high probability when the dimension $n$ is sufficiently large, as proven in Theorem~\ref{thm:POWexactUniform}.
    \end{itemize}
    \item In Section~\ref{sec:computation}, we report the results of our computational experiments and verify several key observations from the previous sections empirically.
    % Theorem~\ref{thm:exactness-kappa=2-stSimplexn=2}: For $n=2$, the convex hull of $\cS^2$ is second-order conic representable.
    % \item Theorem~\ref{thm:exactness-kappa=2-stSimplexn=3}: For $n=3$, the convex hull of $\cS^2$ is semidefinite representable.
    % \item : For general $n$, we prove that optimizing a linear function over SDP+RLT gives the same value as optimizing over $\cS^2$    for some cases. We conjecture that this is true in general. 
    % \item Theorem~\ref{thm:POWexactUniform}: We show that optimizing a random linear function with iid Unif generated coefficients over  the POW relaxation gives the same value as optimizing over $\cS^\kappa$   for large $n$ with high probability for any $\kappa>1$.
\end{itemize}

\noindent\textbf{Notation}: Throughout the paper, $\RR^n$, $\RR_+^n$, $\mathbb{S}^n$, $\mathbb{S}_+^n$, $\mathbb{CP}^n$ are respectively the set of $n$-dimensional real vectors, the set of  $n$-dimensional nonnegative real vectors, the set of  $n\times n$ real symmetric matrices, the set of  $n\times n$ real positive semidefinite matrices, the set of  $n\times n$ completely positive  matrices, and $e$ is the all-ones vector of appropriate dimension. Also, let $e^j$ be the $j$-th standard unit vector. {Given a square matrix $X$, we denote the vector of diagonal of $X$ by $\diag(X)$.}
  For a set $\SS \subseteq \mathbb{R}^n$, we denote its convex hull and the set of extreme points
  %and interior
  as $\text{conv}(\SS)$ and $\textup{extr}(\SS)$,
  %and $\text{int}(X)$, 
  respectively.

  %We denote the  all-one vector with $e$ and its dimension will always be  clear from the context.

%  We denote the   standard simplex and  the full-dimensional simplex in $\mathbb{R}^n$ by $ \Delta^n :=\{x\in\RR^n_+ : e^Tx = 1\}$ and $\Delta^n_+:=\{x\in\RR^n_+ : e^Tx \le 1\}$, respectively.

%We also denote the projection of $X$ onto the space of $x$-variables as $\text{proj}_x(X)$. 

\section{General Ground Sets}
\label{sec:prelim}

In this section, we focus on a general ground set $\cG$. We first prove that it is NP-Hard to optimize a linear function over the set $\cS^\kappa$ in Section~\ref{sec:complexity}. Then, we propose several conic  relaxations for the set $\cS^\kappa$ in Section~\ref{sec:relaxations}. 

\subsection{Complexity}
\label{sec:complexity}

\begin{proposition}\label{prop:complexityNP-Hard}
    Optimizing a linear function over %the set
    $\cS^\kappa$ is NP-Hard for $\kappa>1$.
\end{proposition}
\begin{proof}
Consider the SUBSET-SUM Problem, which is known to be NP-Hard~\cite{garey2002computers}:
 Given $a\in\ZZ^n_{+}$ and  $b\in\ZZ_+$, does there exist  $J \subseteq\{1,\dots,n\}$ such that $\sum_{j\in J} a_j = b$?

 Consider the following problem, which minimizes a linear function over the set~$\cS^\kappa$:
 \[
 (Q): 
 \min_{x\in[0,1]^n,y\in\RR^n} \left\{\sum_{j=1}^n (x_j - y_j) : a^Tx=b, \  y_j=x_j^\kappa, j=1,\dots,n \right\}.
 \]
 Notice that a SUBSET-SUM instance is feasible if and only if the optimal value of~$(Q)$ is zero. Hence, the result follows.
 \end{proof}
 
% \begin{theorem}\label{thm:complexity}
%     Optimizing a linear 
% \end{theorem}

Although optimizing a linear function over the set $\cS^\kappa$ is NP-Hard in general, there is a known polynomially solvable case due to~\cite{bomze2012separable} when 
$\kappa=2$ and $(a,b)=(e,1)$, that is, $\cG=\Delta^n$. We note that the algorithm proposed in this reference can be adapted to any $\kappa>1$.

% \begin{conjecture}\label{conj:optimizeStSimplexkappa}
%     For any $\kappa > 1$ and $\cG=\Delta^n$, optimizing a linear function over $\cS^\kappa$ is polynomial time.
% \end{conjecture}

\subsection{Relaxations}
\label{sec:relaxations}

%\subsection{POW Relaxation}\label{sec:relax-pow}

In this section, we propose several conic  relaxations  for the set $\cS^\kappa$. We start with a definition.
\begin{definition}[Power cone]
    The power cone in $\mathbb{R}^3$ parameterized with $\gamma\in(0,1)$ is defined as
    \[
    \mathcal{C}^{\gamma} := \{ x\in\mathbb{R}_+^2 \times \mathbb{R}: x_1^\gamma x_2^{1-\gamma} \ge |x_3| \}.
    \]
\end{definition} 
\noindent
For example, the set defined by inequalities $y \ge x^\kappa ,  y \ge 0$ is equivalent to $(y, 1, x) \in \mathcal{C}^{1/\kappa} $ for $\kappa > 1$.
%In the special case $\kappa=2$, this set is also SOC.
In the special case with $\kappa=2$, $\mathcal{C}^{1/2}$ is the rotated second-order cone in $\RR^3$.

The simplest relaxation we construct for the set $\cS^\kappa$ involves applying a \textit{single row relaxation} for each nonconvex constraint separately and it utilizes the power cone as defined above to convexify each constraint. 
We will call this relaxation as \textit{the power cone relaxation} (or the $\mathbf{P}$ relaxation, in short) and define it as below:
\begin{equation*}\label{eq:pow}
        \cS^\kappa_P = \{ (x,y)\in \cG\times\RR^n : x_j \ge y_j \ge x_j^\kappa  , j=1,\dots,n  \}.
\end{equation*}
The $\mathbf{P}$ relaxation can be further strengthened using other approaches once a lifted matrix variable $X=xx^T \in \mathbb{S}^n$ is used. Below, we will introduce three of them.

% \begin{equation*}\label{eq:socp}
%         \cS_p = \{ (x,y)\in[0,1]^n\times[0,1]^n : y_j \ge x_j^2 \ge y_j, j=1,\dots,n, \ Ax = b, \ Cx \le d  \},
% \end{equation*}
The first approach to strengthen the $\mathbf{P}$ relaxation  utilizes RLT~\cite{sherali2013reformulation}. For this purpose, let us consider the following set of constraints obtained by using the first level in the RLT hierarchy
% \begin{equation*}\label{eq:rlt}
%         \cS_R  = \left\{ (x,y)\in\cG\times\RR^n : \exists X \in \mathbb{S}^n:   \  \eqref{eq:rlt-all} \right \},
% \end{equation*}
% where
\begin{subequations}\label{eq:rlt-all}
\begin{align}
& \max\{x_i+x_j-1,0\} \le X_{ij} \le \min\{x_i,x_j\} &1&\le i \le j \le n \label{eq:rlt-mccormick} \\
& AX - bx^T = 0 , \ \ Axe^T -AX -be^T +b x^T = 0 \label{eq:rlt-bounds-equality} \\
& CX - dx^T \le 0 , \ \, Cxe^T -CX -de^T + d x^T \le  0 \label{eq:rlt-bounds-inequality} \\
& AXA^T - Axb^T - bx^TA^T + bb^T = 0 \label{eq:rlt-equality-equality}\\
& CXC^T - Cxd^T - dx^TC^T + dd^T \ge  0 \label{eq:rlt-inequality-inequality} \\
& AXC^T - Axd^T - bx^T C^T +bd^T = 0 \label{eq:rlt-equality-inequality} .  
\end{align}
\end{subequations}
 
% \begin{subequations}\label{eq:rlt-all-nonconvex versions}
% \begin{align}
% &   X_{ij}  = x_ix_j &1&\le i \le j \le n \\
% & (Ax-b)x^T = 0 , (Ax-b)(e-x)^T = 0 \\
% & (Cx-d)x^T \le 0 , (Cx-d)(e-x)^T \le 0 \\
% & (Ax-b)(Ax-b)^T = 0 \\
% & (Cx-d)(Cx-d)^T \ge 0 \\
% & (Ax-b)(Cx-d)^T \le 0 
% \end{align}
% \end{subequations}

The second approach to strengthen the $\mathbf{P}$ relaxation  utilizes semidefinite programming and involves the following
   constraints:
\begin{equation}\label{eq:sdp}
     %   \cS_S  =
           \ X \ge 0,
        \   \begin{bmatrix}
            X & x \\
            x^T & 1
        \end{bmatrix} \succeq 0 .
\end{equation}
% Notice that $\cS_p^2 \subseteq \cS_s $. 

The third approach to strengthen the $\mathbf{P}$ relaxation  utilizes second-order cone programming and involves adding the positive semidefinite conditions for the $2\times2$ minors of the matrix $\begin{bmatrix}
            X & x \\
            x^T & 1
        \end{bmatrix}$ only.
For this purpose, we consider the   constraints
\begin{equation}\label{eq:socp}
\begin{split}
     %   \cS_s  = \big\{ (x,y)\in\cG\times\RR^n : \exists 
        & \ X \ge 0,
        \  x_j^2 \le X_{j} \le x_j, j=1,\dots,n , \  X_{ij}^2 \le X_{ii}X_{jj}, \ 1\le i < j \le n .
     %   \big \}.
        \end{split}
\end{equation}
% Notice that $\cS_p^2 \subseteq \cS_s $. 

% \diag(Y) = y,  
In order to make the connection between the $\mathbf{P}$ relaxation  and the three sets of constraints defined above (i.e., constraints~\eqref{eq:rlt-all},~\eqref{eq:sdp} and~\eqref{eq:socp}) stronger, we make a simple observation: Note that we have $X_{jj}=x_j^2$ and $y_j=x_j^\kappa$, implying that $X_{jj}=y_j^{2/\kappa}$. This observation allows us to relate $y_j$ variables with $X_{jj}$ variables through the following constraints:
\begin{equation}\label{eq:transformation}
        \begin{cases}
          y_j \ge X_{jj} \ge y_j^{2/\kappa}, j=1,\dots,n   & \text{ if } \kappa < 2 \\
   X_{jj} = y_j, j=1,\dots,n   & \text{ if } \kappa = 2 \\
           X_{jj} \ge y_j \ge  X_{jj}^{\kappa/2}, j=1,\dots,n 
           & \text{ if } \kappa > 2 \\
        \end{cases} .
\end{equation} 

%\textcolor{red}{TO discuss a bit} 
We are ready to formally introduce the six relaxations we define for the set~$\cS^\kappa$:
\begin{itemize}
    \item The \textbf{P} relaxation:  $ \cS^\kappa_P$
    \item The \textbf{PR} relaxation: $ \cS^\kappa_{P,R}:= \{ (x,y)\in\cS^\kappa_P : \exists X\in\mathbb{S}^n : \eqref{eq:rlt-all}, \eqref{eq:transformation}\} $
    \item The \textbf{PS} relaxation: $ \cS^\kappa_{P,S}:=  \{ (x,y)\in\cS^\kappa_P : \exists X\in\mathbb{S}^n : \eqref{eq:sdp}, \eqref{eq:transformation}\} $ 
    \item The \textbf{Ps} relaxation: $ \cS^\kappa_{P,s}:= \{ (x,y)\in\cS^\kappa_P : \exists X\in\mathbb{S}^n : \eqref{eq:socp}, \eqref{eq:transformation}\} $
    \item The \textbf{PRS} relaxation: $ \cS^\kappa_{P,R,S}:=   \{ (x,y)\in\cS^\kappa_P : \exists X\in\mathbb{S}^n : \eqref{eq:rlt-all},\eqref{eq:sdp}, \eqref{eq:transformation}\} $ 
    \item The \textbf{PRs} relaxation: $ \cS^\kappa_{P,R,s}:=\{ (x,y)\in\cS^\kappa_P : \exists X\in\mathbb{S}^n : \eqref{eq:rlt-all},\eqref{eq:socp}, \eqref{eq:transformation}\} $
\end{itemize}

Let us now discuss some basic properties of these relaxations. The lemma below characterizes the extreme points of the   \textbf{P} relaxation:

\begin{lemma}\label{lem:extrPOW}
    Let $(\hat x,\hat y)\in\extr(\cS^\kappa_P)$. Then, we have $\hat y_j\in\{\hat x_j, \hat x_j^\kappa\}$, $j=1,\dots,n$.
\end{lemma}
Interestingly, adding the SDP constraint~\eqref{eq:sdp}   to the     \textbf{P} relaxation does not further strengthen this relaxation as proven below.
%The following result shows that the POW+SDP relaxation is equal to the POW relaxation.
\begin{proposition}\label{prop:POW=POW+SDP}
Let $\kappa>1$. Then,
    $\cS^\kappa_P = \cS^\kappa_{P,S}$.
\end{proposition}
\begin{proof}
    Note that $\cS^\kappa_P \supseteq \cS^\kappa_{P,S}$. Suppose that $(\hat x, \hat y )\in\extr(\cS^\kappa_P)$. In order to prove the assertion of the proposition, it suffices to show that $(\hat x,\hat y)\in\cS^\kappa_{P,S}$. Due to Lemma~\ref{lem:extrPOW}, we know that  $\hat y_j\in\{\hat x_j, \hat x_j^\kappa\}$, $j=1,\dots,n$. 
    Now, consider the matrix $\hat X := \hat x \hat x^T + D$, where $D\in\mathbb{S}^n$ is a diagonal matrix with entries
    \[
        D_{jj} =
    \begin{cases}
        0 & \text{ if } \hat y_j = \hat x_j^\kappa \\
        \hat y_j-\hat x_j^2 & \text{ if } \hat y_j = \hat x_j \text{ and } \kappa\le2 \\
        \hat y_j^{2/\kappa}-\hat x_j^2 & \text{ if } \hat y_j = \hat x_j \text{ and } \kappa>2 
    \end{cases}.
    \] Observe that both $ \hat x \hat x^T$ and $ D$ are doubly nonnegative matrices and $(\hat x, \hat y,\hat X) $ satisfies~\eqref{eq:transformation}. Hence, the result follows.
\end{proof}

\begin{corollary}\label{cor:POW=POW+SOC}
Let $\kappa > 1$. Then,
    $\cS^\kappa_P = \cS^\kappa_{P,s}$.
\end{corollary}
Due to Proposition~\ref{prop:POW=POW+SDP} and Corollary~\ref{cor:POW=POW+SOC}, we will not study the   \textbf{PS} relaxation and the   \textbf{Ps} relaxation.

Unlike the SDP constraint~\eqref{eq:sdp}, adding the RLT constraints~\eqref{eq:rlt-all} strengthens the   \textbf{P} relaxation. Perhaps more interestingly, the addition of the SDP constraint~\eqref{eq:sdp} and the RLT constraints~\eqref{eq:rlt-all} together further strengthens  the   \textbf{PR} relaxation. The proposition below formalizes this statement.
\begin{proposition}\label{prop:relaxOrderStrict}
We have 
    $\cS^\kappa_P \supseteq \cS^\kappa_{P,R} \supseteq \cS^\kappa_{P,R,s} \supseteq \cS^\kappa_{P,R,S}$, and all the set containment relations can be strict.
\end{proposition}
We prove this proposition in Section~\ref{sec:setContainmentExamples} with suitable examples with  $\cG=\Delta^n$ and $\kappa=2$.

% \begin{itemize}
%     \item POW:  $ \cS^2_p$
% %    \item RLT:  $ \cS_r \cap\{ (x,y)\in\cG\times\RR^n : x\ge y\}$ {\color{blue} we may simply omit this set}
%     \item POW+RLT: $ \cS^2_{p,r}:= \cS^2_p \cap \cS_r \cap \cS_t^2 \cap \{ (x,y)\in\cG\times\RR^n : \exists X \in \mathbb{S}^n: X_{ij}^2 \le X_{ii}X_{jj}, \ 1\le i < j \le n \} $
%     \item POW+SDP: $ \cS^2_{p,s}:= \cS^2_p \cap \cS_s \cap \cS_t^2 := \cS_s \cap \cS_t^2$   (since $ \cS^2_s \subseteq  \cS^2_p$)
%     \item POW+RLT+SDP: $ \cS^2_{p,r,s}:= \cS^2_p \cap \cS_r\cap \cS_s \cap \cS_t^2  $
% \end{itemize}
% %By construction, we have that $ \cS^2_s \subseteq  \cS^2_p$. also principle 2by2

%\subsection{Case with $\kappa = 2$}\label{sec:relax-kappa=not2}

%%%%%%%%%%%%%%%%%%%%%%%%%%%%%%%%%%%%%%%%%%%%%%%%%%%%%%%%%%%%%%%%%%%%

\section{The Standard Simplex as the Ground Set}
\label{sec:main}

% One way to prove the relation  $\conv(\cS^\kappa) = \cT$ is to show  that
% \[
% \min_{(x,y)\in \cS^\kappa}\{ \alpha^Tx+\beta^Ty\} = 
% \min_{(x,y)\in \cT}\{ \alpha^Tx+\beta^Ty\} , 
% \]
% for any $(\alpha,\beta)\in\RR^n\times\RR^n$ such that $ \|(\alpha,\beta)\|=1$. 

Let $\cT$ be a relaxation of the nonconvex set $\cS^\kappa$.
We will use two metrics to measure the quality of the relaxation~$\cT$ for the nonconvex set~$\cS^\kappa$:
\begin{itemize}
    \item Distance-based: 
    $$D_{\cS^\kappa,\cT} := \max_{ (\hat x, \hat y)\in \cT  } \min_{ (x,y)\in \cS^\kappa} \| (x-\hat x, y-\hat y) \| .$$ {This measure quantifies the distance of the farthest point in $\cT$ from $\cS^\kappa$.}
    \item Objective function-based: Let $(\alpha,\beta)\in\RR^n\times\RR^n$.% such that $ \|(\alpha,\beta)\|=1$.
    $$ O_{\cS^\kappa,\cT} (\alpha,\beta) :=  \min_{(x,y)\in \cS^\kappa}\{ \alpha^Tx+\beta^Ty\} - 
    \min_{(x,y)\in \cT}\{ \alpha^Tx+\beta^Ty\} .$$ {This measure quantifies the additive gap between optimizing over $\cT$ and $\cS^\kappa$.}
\end{itemize}
\begin{remark}\label{rem:optLinear=convHull}
    Notice that if $O_{\cS^\kappa,\cT} (\alpha,\beta) = 0$ for all  $(\alpha,\beta)$,  %such that  $ \|(\alpha,\beta)\|=1$, 
then we conclude that $\cT = \conv(\cS^\kappa) $.
\end{remark}

A simple result that we will frequently use in the rest of the paper is the following proposition.  %{\color{red}this is not necessarily the best place to put it..}
\begin{proposition}\label{prop:optimizeOverPOW}
  $  \min_{(x,y)\in \cS^\kappa_P}\{ \alpha^Tx+\beta^Ty\} 
  = \min_{x\in\cG}\{ \sum_{j:\beta_j>0} (\alpha_j x_j + \beta_j x_j^\kappa) + \sum_{j:\beta_j\le0} (\alpha_j+ \beta_j)  x_j \}.
  $
\end{proposition}
Note that the proof of Proposition~\ref{prop:optimizeOverPOW} immediately follows from the extreme point description of the set  $\extr(\cS^\kappa_P)$ (the feasible region of the \textbf{P} relaxation)
given in Lemma~\ref{lem:extrPOW}.
% \begin{proof}
% \end{proof}

%\subsection{Standard simplex as the ground set}
%\label{sec:standardSimplex}

The remainder of this section is organized as follows. In Section~\ref{sec:setContainmentExamples}, we prove Proposition~\ref{prop:relaxOrderStrict} for the case $\cG = \Delta^n$. Section~\ref{sec:exactness-kappa=2-stSimplex} presents several cases where the exact convex hull can be obtained when $\kappa = 2$. In Section~\ref{sec:approxgenkappa}, we establish results that assess the strength of $\cS_P^\kappa$ in approximating $\cS^\kappa$, with our main result showing that $\cS_P^\kappa$ provides an increasingly tight relaxation as $n \to \infty$. 
Section~\ref{sec:rpt} introduces another class of conic-representable inequalities for $\cS^\kappa$, derived using RPT. %%--%% a technique analogous to RLT. 

%Finally, Section~\ref{sec:otherground} reports observations for two other classes of ground sets.

\subsection{Set containment examples}
\label{sec:setContainmentExamples}

We now prove the strict set inclusion relations stated in Proposition~\ref{prop:relaxOrderStrict}.
We start with a lemma, which simplifies the RLT constraints when the ground set is the standard simplex.
\begin{lemma}\label{lem:RLTforStandardSimplex}
Let $\cG=\Delta^n$. Then, %$ \{x \in \Delta^n, X \in \mathbb{S}^n : \eqref{eq:rlt-all} \} =  \{x \in \Delta^n, X \in \mathbb{S}^n : \eqref{eq:rlt-stSimplex} \}$
\begin{equation*}
    \begin{split}
        %\cS_R \cap \cS_T^\kappa
    %     \cS^\kappa_{P,R}:= 
       %  \{ (x,y)\in\cS^\kappa_P : \exists X\in\mathbb{S}^n : \eqref{eq:rlt-all}, \eqref{eq:transformation}\} = 
      %  &  \left\{ (x,y)\in\cS^\kappa_P : \exists X \in \mathbb{S}^n: X \ge 0, \  Xe = x , \ \eqref{eq:transformation}\right \}  .
        %
        \{x \in \Delta^n, X \in \mathbb{S}^n : \eqref{eq:rlt-all} \} =  \{x \in \Delta^n, X \in \mathbb{S}^n : \eqref{eq:rlt-stSimplex} \},
    \end{split}
\end{equation*}
where
\begin{equation}\label{eq:rlt-stSimplex} 
    X \ge 0, \ Xe = x.
\end{equation}
Moreover, if $n \le 3$ and $\kappa=2$, we have 
\begin{equation*}
    \begin{split}
       % \cS_R \cap \cS_T^\kappa 
         \cS^\kappa_{P,R}
        & =  \left\{ (x,y)\in\cS^\kappa_P  :  \ x_j-y_j \le  \sum_{i\neq j } (x_i-y_i)  ,   \ j=1,\dots,n  \right \}  .
    \end{split}
\end{equation*}
\end{lemma}
\begin{proof}
For $\cG=\Delta^n$,  we have $A=e^T$, $b=1$ and $C=d=0$.

In order to prove the first result, notice that equations \eqref{eq:rlt-bounds-equality} give the equation $Xe=x$. Together with the condition $x\in\Delta^n$, we observe that  
     the McCormick envelopes given in~\eqref{eq:rlt-mccormick} are redundant except $X\ge0$. Therefore, the first result~follows.

In order to prove the second result for $n=2$,   observe that we have
\begin{equation*}
\begin{split}
\cS^2_{P,R}  &= \{ (x,y)\in\cS^2_P : \exists X \in \mathbb{S}^2:  X \ge 0,  \ Xe=x, \ \diag(X) = y \} \\
&=  \{ (x,y)\in\cS^2_P : \exists X \in \mathbb{S}^2:  X \ge 0,  \ y_1 + X_{12} = x_1, X_{12} + x_2 = y_2 \} \\
&=  \{ (x,y)\in\cS^2_P : x_1-y_1 = x_2-y_2   \} .
\end{split}
\end{equation*}
     
In order to prove the second result for $n=3$, observe that we have
\begin{equation*}
\begin{split}
\cS^2_{P,R}  & = \{ (x,y)\in\cS^2_P : \exists X \in \mathbb{S}^3: \ X \ge 0,  \ Xe=x, \diag(X) = y  \} \\
&=  \{ (x,y)\in\cS^2_P :  
\exists X \in \mathbb{S}^3: \ X \ge 0, \ y_1+X_{12}+X_{13} = x_1, \\
& \hspace{5.0cm} X_{12}+y_2+X_{23} = x_2, 
X_{13}+X_{23}+y_3 = x_3   \}  \\
&=  \{ (x,y)\in\cS^2_P : 
(x_3-y_3) - (x_1-y_1) - (x_2-y_2) \le 0, \\
&\hspace{3.8cm} (x_2-y_2) - (x_1-y_1) - (x_3-y_3) \le 0, \\
&\hspace{3.8cm} (x_1-y_1) - (x_2-y_2) - (x_3-y_3) \le 0   \} ,
\end{split}
\end{equation*}
where the last equality follows from the fact that the values $X_{12}$, $X_{13}$ and $X_{23}$ are uniquely determined given $x$ and $y$.
\end{proof}

\begin{remark}
    By virtue of the first result in Lemma~\ref{lem:RLTforStandardSimplex}, we can replace the generic RLT constraints~\eqref{eq:rlt-all} with constraints~\eqref{eq:rlt-stSimplex} when $\cG=\Delta^n$ for  \textbf{PR},  \textbf{PRS} and  \textbf{PRs}  relaxations. More precisely, we  have the following: 
\begin{itemize}
    \item  $ \cS^\kappa_{P,R}:= \{ (x,y)\in\cS^\kappa_P : \exists X\in\mathbb{S}^n : \eqref{eq:rlt-stSimplex}, \eqref{eq:transformation}\} $
    \item  $ \cS^\kappa_{P,R,S}:=   \{ (x,y)\in\cS^\kappa_P : \exists X\in\mathbb{S}^n : \eqref{eq:rlt-stSimplex},\eqref{eq:sdp}, \eqref{eq:transformation}\} $ 
    \item  $ \cS^\kappa_{P,R,s}:=\{ (x,y)\in\cS^\kappa_P : \exists X\in\mathbb{S}^n : \eqref{eq:rlt-stSimplex},\eqref{eq:socp}, \eqref{eq:transformation}\} $
\end{itemize}
\end{remark}

The following result shows that the \textbf{PR} relaxation can be strictly stronger than the \textbf{P} relaxation.
\begin{proposition}\label{prop:POW=not=POW+RLT}
The relation
    $\cS^\kappa_P \supseteq \cS^\kappa_{P,R}$ can be strict.
\end{proposition}
\begin{proof}
    Let $\kappa=2$ and consider $\cG=\Delta^2$. 
    %In this case, we have  $\cS_R = \{ (x,y)\in\Delta^2\times\RR^2 : \exists X \in \SS^2:   \ Xe=x \} =  \{ (x,y)\in\Delta^2\times\RR^2 : x_1-y_1 = x_2-y_2 \} $. 
 %   Consider the point $$(\hat x, \hat y)=(1/2,1/2,1/2,1/4).$$
    Consider the point $(\hat x, \hat y)$, where 
    $$ \hat x=\begin{bmatrix}1/2 \\ 1/2 \end{bmatrix} \text{ and } \hat y=\begin{bmatrix}1/2 \\1/4\end{bmatrix}.$$  
    Due to Lemma~\ref{lem:RLTforStandardSimplex}, we deduce that $(\hat x, \hat y)\in \cS^\kappa_P$ but $(\hat x, \hat y)\not\in \cS^\kappa_{P,R}$. 
\end{proof}

The following result shows that the \textbf{PRs} relaxation can be strictly stronger than the \textbf{PR} relaxation.
\begin{proposition}\label{prop:POW+RLT=not=POW+RLT+SOC}
The relation
    $\cS^\kappa_{P,R} \supseteq \cS^\kappa_{P,R,s}$ can be strict.
\end{proposition}
\begin{proof}
    Let $\kappa=2$  and consider $\cG=\Delta^3$.     
    Consider the point $(\hat x, \hat y)$, where 
    $$ \hat x=\begin{bmatrix}1/3 \\ 1/3 \\ 1/3 \end{bmatrix} \text{ and }
    \hat y=\begin{bmatrix}1/9 \\ 1/9 \\1/3\end{bmatrix}.$$   
    Due to Lemma~\ref{lem:RLTforStandardSimplex}, we deduce that this point belongs to $\cS^\kappa_{P,R}$ with the following \textit{unique} selection of~$\hat X$:
    \[
    \hat X =\begin{bmatrix}
        1/9 & 2/9 & 0 \\
        2/9 & 1/9 & 0 \\
        0 & 0 & 1/3 \\
    \end{bmatrix}.
    \]
    However,  $(\hat x, \hat y)\not\in \cS^\kappa_{P,R,s}$ since the inequality $X_{11}X_{22} \ge X_{12}^2$ is violated.
\end{proof}

The following result shows that the \textbf{PRS} relaxation can be strictly 
stronger than the \textbf{PRs} relaxation.
\begin{proposition}\label{prop:POW+RLT+SOC=not=POW+RLT+SDP}
The relation
    $\cS^\kappa_{P,R,s} \supseteq \cS^\kappa_{P,R,S}$ can be strict.
\end{proposition}
\begin{proof}
    Let $\kappa=2$  and consider $\cG=\Delta^3$. 
 %   Consider the point $$(\hat x, \hat y)=(1/2,1/3,1/6,1/4,1/8,1/30).$$   
    Consider the point $(\hat x, \hat y)$, where 
    $$ \hat x=\begin{bmatrix}1/2 \\ 1/3 \\ 1/6 \end{bmatrix} \text{ and }
    \hat y=\begin{bmatrix}1/4 \\ 1/8 \\1/30\end{bmatrix}.$$  
    Due to Lemma~\ref{lem:RLTforStandardSimplex}, we deduce that this point belongs to $\cS^\kappa_{P,R}$ with the following \textit{unique} selection of~$\hat X$:
    \[
  \hat X = \frac1{240} \begin{bmatrix}
        60 & 39 & 21 \\
        39 & 30 & 11 \\
        21 & 11 & 8 \\
    \end{bmatrix}.
    \]
    In addition, its principal $2\times2$ minors, which are $\frac{279}{240^2}$, $\frac{39}{240^2}$, $\frac{119}{240^2}$, are all nonnegative, hence, $(\hat x, \hat y)\in \cS^\kappa_{P,R,s}$. However, its determinant is $-\frac{1}{240^2}<0$, hence, $\hat X$ is not positive semidefinite. Therefore, $(\hat x, \hat y)\not\in \cS^\kappa_{P,R,S}$. 
\end{proof}

\subsection{Exactness Results for $\kappa=2$}
\label{sec:exactness-kappa=2-stSimplex}

We now provide some exactness results for $\kappa=2$. Theorem~\ref{thm:exactness-kappa=2-stSimplexn=2} states that the convex hull of $\cS^\kappa$ is second-order cone representable for $n=2$, which is obtained from the \textbf{PR} relaxation.
\begin{theorem}\label{thm:exactness-kappa=2-stSimplexn=2}
    Let $\cG=\Delta^2$ and $\kappa = 2$. Then, $\conv(\cS^\kappa) = \cS^\kappa_{P,R} $.
\end{theorem}
\begin{proof}
It suffices to show that $\conv(\cS^\kappa) \supseteq \cS^\kappa_{P,R} $. Let $(\hat x, \hat y)\in \cS^\kappa_{P,R}$. Then, we have $e^T \hat x=1$, $ \hat x_j\ge \hat y_j\ge \hat x_j^2 $ and $\hat x_1-\hat y_1 = \hat x_2-\hat y_2 $, where the last equality follows from Lemma~\ref{lem:RLTforStandardSimplex}. 

\noindent
    {Case 1}: $\hat y_1 = 0$. In this case,   we conclude that 
    \[
    \hat y_1 = 0 \implies 
    \hat x_1 = 0 \implies 
    \hat x_2 = 1 \implies 
    \hat y_2 = 1.
    \]
    Notice that $(\hat x, \hat y) \in \mathcal{S}^\kappa$, hence, $(\hat x, \hat y) \in \conv(\mathcal{S^\kappa})$.

\noindent
    {Case 2}: $\hat y_1 > 0$. In this case, we also have $\hat x_1 > 0$. We claim that
    \[
    \begin{bmatrix}
       \hat x_1 \\ \hat x_2 \\ \hat y_1 \\ \hat y_2  
    \end{bmatrix}
        = (1-\lambda)
    \begin{bmatrix}
        0 \\ 1 \\ 0 \\ 1
    \end{bmatrix}
    + \lambda 
    \begin{bmatrix}
       \tilde x_1 \\ \tilde x_2 \\ \tilde y_1 \\ \tilde y_2  
    \end{bmatrix},
    \]
    for some $(\tilde x, \tilde y) \in \mathcal{S}^\kappa$ with $\lambda = \frac{\hat x_1^2}{\hat y_1} \in (0,1)$. In fact, 
    we have
    $\tilde y_1 = \frac{\hat y_1^2}{\hat x_1^2} $,
    $\tilde y_2 = \frac{(\hat y_1-\hat x_1)^2}{\hat x_1^2}$, 
    $ \tilde x_1 = \frac{\hat y_1}{\hat x_1} $ and
    $\tilde x_2 = \frac{\hat x_1-\hat y_1}{\hat x_1}$. Then, it is straightforward to check that $(\tilde x, \tilde y) \in \mathcal{S}^\kappa$, hence, we conclude that  $(\tilde x, \tilde y) \in \conv(\mathcal{S}^\kappa)$.
 \end{proof}

For $n>2$, the convex hull of $\cS^\kappa$ is harder to characterize. The following fact provides the convex hull using the intractable completely positive cone due to~\cite{burer2009copositive}.
\begin{theorem} \label{thm:burerCONV}
Let $\cG=\Delta^n$ and $\kappa=2$. Then, we have
$$\conv(\cS^\kappa) = \left\{ (x,y) \in\Delta^n\times\rr^n: \ \exists X \in \mathbb{S}^n : \begin{bmatrix}
     X &  x \\  x^T & 1
\end{bmatrix}\in \CP^{n+1}, \ \eqref{eq:rlt-stSimplex} \right\}. $$
\end{theorem} 
In our paper, we look for tractable relaxations, for example, the \textbf{PRS} relaxation, which is a doubly nonnegative relaxation. Since we analyze the case where $\cG=\Delta^n$ and $\kappa=2$ in this part, observe that optimizing a linear function $\alpha^Tx+\beta^Ty$ over the \textbf{PRS} relaxation is equivalent to the following optimization problem, which we call as the \textbf{PRS'} relaxation:
\begin{equation*}
    \min_{x\in\Delta^n, X\in \mathbb{S}^n}\left\{ \alpha^Tx + \beta^T \diag(X) : \ \begin{bmatrix}
     X &  x \\  x^T & 1
\end{bmatrix}\in \mathbb{S}_+^{n+1}, \ \eqref{eq:rlt-stSimplex} \right\}.
\end{equation*}

The following corollary is a consequence of Theorem~\ref{thm:burerCONV}.
%We start our analysis with a lemma.
\begin{corollary}\label{cor:CPP}
The following hold true: 
\begin{enumerate}[(i)]
\item 
Let  $\cG=\Delta^3$ and $\kappa = 2$. Then, $\conv(\cS^\kappa) = \cS^\kappa_{P,R,S} $.  
    \item 
    Consider an optimal solution $(\hat x,   \hat X)$ to the \textbf{PRS'} relaxation and assume that the matrix $\begin{bmatrix}
   \hat  X & \hat x \\ \hat x^T & 1
\end{bmatrix} \in \CP^{n+1} $. Then,  $\OO_{\cS^\kappa,\cS^\kappa_{P,R,S}} (\alpha,\beta)=0$. 
\end{enumerate}
\end{corollary}
{Note that part (i) of Corollary~\ref{cor:CPP} follows from the fact that up to $4\times 4$ matrices, doubly non-negative matrices are also completely positive~\cite{maxfield1962matrix}. By the same token, we have $\cG=\Delta^2$ and $\kappa = 2$, then $\conv(\cS^\kappa) = \cS^\kappa_{P,R,S} $. 
Theorem~\ref{thm:exactness-kappa=2-stSimplexn=2} is a strengthening, since it says that we do not need the PSD constraints to obtain the convex hull when $n =2$.}

%-%-%-
%%%    \item The \textbf{PRS} relaxation: $ \cS^\kappa_{P,R,S}:=   \{ (x,y)\in\cS^\kappa_P : \exists X\in\mathbb{S}^n : \eqref{eq:rlt-all},\eqref{eq:sdp}, \eqref{eq:transformation}\} $ 
%%
%-%-%-
 
% \begin{theorem}\label{thm:exactness-kappa=2-stSimplexn=3}
    
% \end{theorem}
% \begin{proof}
%     {\color{red}DNN relaxation is tight for $n\le3$ (cite Burer \cite{burer2009copositive}) since CPP=DNN up to dimension 4. - to be made more precise}
% \end{proof}

Next, in Theorem~\ref{thm:exactness-kappa=2-specialCases}, we present three sufficient conditions on $\beta$ so that we obtain $\OO_{\cS^\kappa,\cS^\kappa_{P,R,S}} (\alpha,\beta) = 0$. 
% Before stating this result formally, we introduce some further notation.
% Let us denote the principal submatrix of 
% $\begin{bmatrix}
%    \hat  X & \hat x \\ \hat x^T & 1
% \end{bmatrix} \in \mathbb{S}^{n+1}$
% obtained from the rows and columns corresponding to an index set $J\subseteq\{1,\dots,n+1\}$ as $M_J$. For $M_J \succ 0$, we will denote the Schur complement of $\begin{bmatrix}
%    \hat  X & \hat x \\ \hat x^T & 1
% \end{bmatrix}$ with respect to  $M_J$ as $N_J$. We also define a matrix $\bar N_J\in\mathbb{S}^{n+1}$, whose entries in $J\times J$ are exactly the same as in the entries of $S_J$, and the entries not in $J\times J$ are zero.  For example, we have $M_{\{n+1\}}=1$, $N_{\{n+1\}}=\hat X-\hat x\hat x^T$ and $\bar N_{\{n+1\}}= \begin{bmatrix}
%     \hat X-\hat x\hat x^T & 0 \\ 0^T & 0
% \end{bmatrix}$.
\begin{theorem}\label{thm:exactness-kappa=2-specialCases}
    Let $\cG=\Delta^n$ and $\kappa = 2$. %Then,  $\OO_{\cS^\kappa,\cS^\kappa_{P,R,S}} (\alpha,\beta)=0$ if any of the following hold true:
    Then, we have $O_{\cS^\kappa,\cS^\kappa_{P,R,S}} (\alpha,\beta) =0$ under any of the following cases:
    \begin{itemize}
   %     \item $n\le3$.
        \item Case 1: $|\{j=1,\dots,n : \beta_j \le 0\}| = n$.    
        \item Case 2: $|\{j=1,\dots,n : \beta_j \ge 0\}| = n$.
        \item Case 3: $|\{j=1,\dots,n : \beta_j > 0\}| = 1$.\end{itemize}
\end{theorem}
\begin{proof}
% We will ``prove'' this ``theorem'' by showing that minimizing a linear function $\alpha^T x + \beta^T y$ yields the same optimal value over $\mathcal{S}^\kappa$ and $\mathcal{S}^\kappa_{P,R,S}$. Let $  J^-  := \{j  : \beta_j \le 0\} $ and $  J^+  := \{j : \beta_j > 0\} $. 

    Consider an optimal solution $(\hat x,   \hat X)$ to the \textbf{PRS'} relaxation. The assertion of the theorem follows if the set $J_< = \{j : \hat X_{jj} >\hat x_j^2\} = \emptyset$ since in this case $\hat X = \hat x \hat x^T$. In the rest of the proof, let us assume that $J_< \neq \emptyset$ for every optimal solution to the \textbf{PRS'} relaxation.

    We first prove two preliminary results: 
    \begin{claim}\label{claim1}
        For $i\not\in J_<$, we have  $\hat X_{ij}=\hat x_i \hat x_j$ for any $j=1,\dots,n$.
    \end{claim}
    \begin{proof}
        Consider positive semidefinite matrix $\hat X - \hat x \hat x$. Note that if a diagonal entry of this matrix is zero, i.e., $\hat X_{ii}-\hat x_i^2=0$ for some $i$, then we have that the off-diagonal entries $\hat X_{ij}-\hat x_i \hat x_j=0$  for every $j=1,\dots,n$. Hence, the result follows.
    \end{proof}  
    
    \begin{claim}\label{claim4}
          If there exist $i,j\in J_<$ such that $\beta_i + \beta_j < 0$, then $\hat X_{ij}=0$.
    \end{claim}
    \begin{proof}  Suppose not. Then, we can {construct} 
    %switch
    a new solution $\tilde X$ that agrees with $\hat X$ on each entry except: %that 
    \[
    \begin{bmatrix}
        \tilde X_{ii} & \tilde X_{ij} \\
        \tilde X_{ij} & \tilde X_{jj} \\
    \end{bmatrix}
    =
    \begin{bmatrix}
        \hat X_{ii} & \hat X_{ij} \\
        \hat X_{ij} & \hat X_{jj} \\
    \end{bmatrix}
    +
    \begin{bmatrix}
        \hat X_{ij} & -\hat X_{ij} \\
        -\hat X_{ij} & \hat X_{ij} \\
    \end{bmatrix}.
    \]
    Notice that $\tilde X$ is obtained from   $\hat X$ by a diagonally dominant shift, hence, it is positive semidefinite. In addition, it maintains the row sums being the same as $\hat X$. Therefore, it is a feasible solution. However, the objective function difference between   $(\hat x,  \hat X)$ and $(\hat x,  \tilde X)$ is $(\beta_i \hat X_{ii} + \beta_j \hat X_{jj}) - (\beta_i \tilde X_{ii} + \beta_j \tilde X_{jj}) = -(\beta_i+\beta_j)\hat X_{ij} > 0$. However, this is a contradiction to  $(\hat x,  \hat X)$ being an optimal solution.
    \end{proof}
    In the remainder of the proof, we will assume that $\hat X_{ij}=0$ for $i,j$ such that $\beta_i + \beta_j \le 0$ (notice that if $\beta_i + \beta_j = 0$, we can find a solution that satisfies this property). %In particular, we have $\hat X_{ij}=0$ for $i,j\in J_<^-$.
    
    Now, we will prove each of the three sufficient conditions separately.
    \begin{itemize}
 %       \item This case follows as a consequence of Theorem~\ref{thm:exactness-kappa=2-stSimplexn=3}.
        %
        \item Case 1:
        Due to Claim~\ref{claim4}, we have that $\hat X_{ij}=0$ for $i\neq j$, implying that $\hat X_{jj} =  \hat x_j $. Then, observe that  we have
    \[
    (\hat x, \hat X) = \sum_{j\in J_<} \hat x_j ( e^j,(e^j)(e^j)^T   ).
    \]
    In other words, $(\hat X, \hat x)$ is a convex combination of $|J_<|$ many feasible solutions.  Therefore, at least one of these solutions must have  an objective function value at least as good as $(\hat X, \hat x)$, which is  a contradiction. 
        \item Case 2: 
        Let us consider a new feasible solution  $(\hat x,  \hat x \hat x^T)$ to the \textbf{PRS'} relaxation. Observe that  the objective function difference between   $(\hat x,  \hat X)$ and $(\hat x,   \hat x \hat x^T)$ is $   \sum_{i\in J_<} \beta_i (\hat X_{ii} - \hat x_i^2) \ge 0 $. This implies that $(\hat x,   \tilde X)$ is also an optimal solution, which is   a contradiction. 
        \item 
        {Case 3:
        Let 
        $J_<^- :=\{ j\in J_< : \hat X_{jj}>\hat x_j^2, \ \beta_j \le 0\}$, 
        $J_=^- :=\{ j  : \hat X_{jj}=\hat x_j^2 >0, \ \beta_j \le 0\}$,
        $J_0^- :=\{ j  : \hat X_{jj}=\hat x_j^2 = 0, \ \beta_j \le 0\}$.
        If $J_<^- = \emptyset$, then the statement is trivially true as we reduce to Case 2 with a single variable having a positive $\beta_j$ coefficient.
        Suppose, without loss of generality, that
        $\beta_{n}>0$, $J_<^- = \{1,\dots,k\}$, $J_=^-=\{k+1,\dots,k'\}$ and $J_0^- =\{k'+1,\dots,n-1\}$.
        In the rest of the proof, we will consider the  following submatrix of $\begin{bmatrix}
              \hat X & \hat x \\
              \hat x^T   & 1
        \end{bmatrix}$, which is obtained by deleting the identically zero rows and columns in $J_0^-$:
        \begin{equation}\label{eq:partialSolMatrix}
         \begin{bmatrix}
            \hat X_{11} &  &  &  \hat X_{1n} &\hat x_1 \\
             & \ddots &  & \vdots & \vdots \\
             &    & \hat X_{k',k'}& \hat X_{k',n} & \hat x_{k'} \\
            \hat X_{1n} &\cdots & \hat X_{k',n}& \hat X_{nn} &\hat x_n \\
            \hat x_1 & \cdots & \hat x_{k'} & \hat x_n & 1
        \end{bmatrix}.
        \end{equation}
        Notice that the  Schur complement of this matrix with respect to the first $k'$ rows and columns is obtained as
        \[
        0 \preceq\begin{bmatrix}
          \hat X_{nn}- \sum_{j = 1}^{k'} \frac{\hat X_{jn}^2}{ \hat X_{jj}} &  \hat x_n- \sum_{j = 1}^{k'} \frac{\hat x_j \hat X_{jn}}{ \hat X_{jj}}    
\\
        \hat x_n- \sum_{j = 1}^{k'} \frac{\hat x_j \hat X_{jn}}{ \hat X_{jj}}   & 1- \sum_{j = 1}^{k'} \frac{\hat x_j^2}{ \hat X_{jj}} 
        \end{bmatrix} = \bigg(1- \sum_{j = 1}^{k'} \frac{\hat x_j^2}{ \hat X_{jj}} \bigg)\begin{bmatrix}
            1&1\\1&1
        \end{bmatrix},
        \]
        where the equality follows as a consequence of $\hat X e = \hat x$. Since we have $1- \sum_{j = 1}^{k'} \frac{\hat x_j^2}{ \hat X_{jj}}  \ge 0$, we conclude that $ |J_=^-| \le 1$. Now, we have two subcases:
        \begin{itemize}
        \item $ |J_=^-| = 1$. In this case, we must have $k=0$ and $k'=1$, meaning that $J_<^- = \emptyset$. In this situation, we again reduce to Case 2 with a single variable having a positive $\beta_j$ coefficient.
        \item  $ |J_=^-| = 0$. In this case, we must have $k=k'$.
        Since we have $1- \sum_{j = 1}^k \frac{\hat x_j^2}{ \hat X_{jj}} \ge 0$, we can write an explicit completely positive decomposition of the solution matrix in equation~\eqref{eq:partialSolMatrix} as (recall Corollary~\ref{cor:CPP})
        $$%\begin{bmatrix}
          %    \hat X & \hat x \\
         %     \hat x^T   & 1
        %\end{bmatrix} =
        \sum_{j = 1}^k
        \frac{\hat x_j^2}{\hat X_{jj}}
          \bigg(\frac{\hat X^j}{\hat x_j}\bigg) \bigg(\frac{\hat X^j}{\hat x_j}\bigg)^T +
        \bigg (1- \sum_{j = 1}^k \frac{\hat x_j^2}{ \hat X_{jj}}\bigg)
        (e^n+e^{n+1})(e^n+e^{n+1})^T,
        $$
        where $\bar X^j$ is the $j$-th column of matrix $\bar X$.
        Notice that  the solution matrix on the left-hand side is a convex combination of $k+1$ many feasible solutions.  Therefore, at least one of these solutions must have  an objective function value at least as good as $(\hat X, \hat x)$, which is  a contradiction. 
        \end{itemize}
        }

    \end{itemize}
    \
\end{proof}
We note that results similar to those of Cases 1 and 2 in Theorem~\ref{thm:exactness-kappa=2-specialCases} have also~been shown in~\cite{gokmen2022standard} utilizing conic duality whereas our approach uses only the primal~problem.

%Due to our intuition from 
{Motivated by the insights derived from}
Theorem~\ref{thm:exactness-kappa=2-specialCases} and our extensive computational experiments reported in Section~\ref{sec:comp-results} (see, in particular, Figure~\ref{fig:gapTime-kappa} with $\kappa=2$), we have come up with the following conjecture:
\begin{cnj}\label{conj:exactness-kappa=2-stSimplexn-any}
    Let $\cG=\Delta^n$ and $\kappa = 2$. Then, $\OO_{\cS^\kappa,\cS^\kappa_{P,R,S}}{(\alpha,\beta)}=0$ {for all $\alpha, \beta \in \mathbb{R}^n \times \mathbb{R}^n$,} implying that $\conv(\cS^\kappa) = \cS^\kappa_{P,R,S} $.  
\end{cnj}

\subsection{Approximation results for general 
$\kappa$}\label{sec:approxgenkappa}

%\subsubsection{TO BE FORMATTED : If we optimize a linear function, under what conditions relaxation $\cS^\kappa_P$  is exact?}

We now switch our attention from exactness results to approximation guarantees for general $\kappa$.

\subsubsection{Distance-based approximation results}

Our first distance-based approximation result below gives an upper bound on the worst-case distance between  the nonconvex set  $\cS^\kappa$ and  the set $\cS_P^\kappa$ (i.e., the feasible region of the \textbf{P} relaxation), which increases with both the exponent~$\kappa$ and the dimension~$n$.

\begin{proposition}\label{prop:dbound-stSimplex-ub}
Let {$n \geq 2$},   $\cG=\Delta^n$ and $\kappa>1$. Then, we have
$$D_{\cS^\kappa,\mathcal{S}_{P}^\kappa} := \max_{ (\hat x, \hat y)\in \mathcal{S}_{P}^\kappa} \min_{ (x,y)\in \cS^\kappa} \| (x-\hat x, y-\hat y) \|_1 \le 1 - n^{1-\kappa}  .$$
\end{proposition} 
\begin{proof}
Since we want to establish an upper bound to 
$D_{\cS^\kappa,\mathcal{S}_{P}^\kappa} $, we can upper bound the inner minimization problem. This can be done in the inner optimization problem  by just substituting a feasible solution. 
In particular, we will substitute, ${x}_i :=\hat x_i$ and ${y}_i:=\hat x_i^\kappa$  for all $i=1,\dots,n$, where $\sum_{i=1}^n \hat x_i =1$. Then, we  have
$$\min_{({x}, {y}) \in \mathcal{S}^\kappa} \| x - \hat x\|_1 + \| y - \hat y\|_1 \leq \sum_{i =1}^n |\hat y_i -\hat x_i^\kappa|.$$
Therefore, we are left with solving:
\begin{eqnarray*}
    &\max_{(\hat x,\hat y)\in\Delta^n\times\rr^n} \left\{ \sum_{i = 1}^n |\hat y_i -\hat x_i^\kappa| : \hat x_i ^\kappa \leq\hat y_i \leq\hat x_i , \  i=1,\dots,n \right\}.
\end{eqnarray*}
Since $\hat y_i \geq\hat x^\kappa_i$  for points in $\mathcal{S}_P^\kappa$, we obtain 
\begin{eqnarray*}
        \max_{(\hat x,\hat y) \in \mathcal{S}_P^\kappa} \sum_{i = 1}^n |\hat y_i -\hat x_i^\kappa| &=& \max_{(\hat x,\hat y) \in \mathcal{S}_P^\kappa} \sum_{i = 1}^n (\hat y_i -\hat x_i^\kappa) = \max_{\hat  x\in\Delta^n} \sum_{i = 1}^n (\hat x_i-\hat x_i^\kappa) \\ 
        &=& n \left( \frac1n - \frac{1}{n^\kappa} \right)  
        =  1 - n^{1-\kappa},
\end{eqnarray*}
which completes the proof.
%}
\end{proof}
Notice that $1 - n^{1-\kappa}$ converges to 0 and 1 as $\kappa\to1^+$ and $\kappa\to\infty$, respectively, and it is equal to $1-\frac1n$ for $\kappa=2$.

{How good is the upper bound presented in the above result?}
Our second distance-based approximation result in Proposition~\ref{prop:dbound-stSimplex-lb} gives a matching lower bound on the worst-case distance between  the nonconvex set  $\cS^\kappa$ and its convex hull, which increases with the exponent~$\kappa$ and dimension $n$.
{This lower bound shows that $\textbf{P}$ is indeed a ``reasonable" convex relaxation, since even   $\cS^\kappa$ and its convex hull have distances of similar order between them.}
To prove Proposition~\ref{prop:dbound-stSimplex-lb}, we   need a technical lemma:

%\newpage
\begin{lemma}\label{lem:tech}
Consider the function
$f(x) = \left|x - \frac{1}{n}\right| + \left|x^{\kappa}  - \frac{1}{n}\right| + x$.
If $n \geq 2$ and $\kappa > 1$, then the minimizer of this function is achieved at $x = \frac{1}{n}$. 
\end{lemma}
\begin{proof}
We may write the function as:
\begin{eqnarray*}
    f(x) = \left\{\begin{array}{rl} \frac{2}{n} - x^{\kappa},& x\in [0, \frac{1}{n}]\\
    2x - x^{\kappa},& x\in \left[\frac{1}{n}, \left(\frac{1}{n}\right)^{\frac{1}{\kappa}}\right]\\
    - \frac{2}{n} + 2x + x^{\kappa}, & x\in \left[\left(\frac{1}{n}\right)^{\frac{1}{\kappa}}, \infty\right)\\
    \end{array} \right.
\end{eqnarray*}
We see that the function is continuous, is decreasing in the interval $[0, \frac{1}{n}]$, concave in the interval $\left[\frac{1}{n}, \left(\frac{1}{n}\right)^{\frac{1}{\kappa}}\right]$ and increasing in the interval $\left[\left(\frac{1}{n}\right)^{\frac{1}{\kappa}}, \infty\right)$.
Therefore, the only possible optimal solutions are $\left\{\frac{1}{n}, \left(\frac{1}{n}\right)^{\frac{1}{\kappa}}\right\}$. To complete the proof, we need to show that $f(\frac{1}{n}) \leq f(\left(\frac{1}{n}\right)^{\frac{1}{\kappa}})$ or equivalently:
\begin{eqnarray}
    \frac{2}{n} - \left(\frac{1}{n}\right)^{\kappa} \leq 2\left(\frac{1}{n}\right)^{\frac{1}{\kappa}} - \frac{1}{n} 
   \ \Leftrightarrow \ \frac{1}{n} \leq \frac{2}{3}\left(\frac{1}{n}\right)^{\frac{1}{\kappa}} + \frac{1}{3}\left(\frac{1}{n}\right)^{\kappa}. \label{eq:toprove1}
\end{eqnarray}
Note that for $\kappa =1$, (\ref{eq:toprove1}) holds. The proof will be complete by showing that the function $g(\kappa):= \frac{2}{3}\left(\frac{1}{n}\right)^{\frac{1}{\kappa}} + \frac{1}{3}\left(\frac{1}{n}\right)^{\kappa}$ is non-decreasing in $\kappa$ for $\kappa\geq 1$ and $n\geq 2$.

Observe that 
$
g'(\kappa) = \frac{2}{3}\,n^{-1/\kappa}\frac{\textup{ln}(n)}{\kappa^2} 
- \frac{1}{3} (\textup{ln}(n))\,n^{-\kappa}
= \frac{\textup{ln}(n)}{3}\left(\frac{2\,n^{-1/\kappa}}{\kappa^2} - n^{-\kappa}\right).
$
Since $n \ge 2$ implies $\textup{ln}(n) > 0$, it suffices to show
$
\frac{2\,n^{-1/\kappa}}{\kappa^2}  \geq n^{-\kappa}$. Taking logarithms, this inequality is equivalent to showing
$
\textup{ln}(2) - 2\textup{ln}(\kappa) + (\textup{ln}(n))\left(\kappa - \frac{1}{\kappa}\right)\geq 0.
$ Since $n \geq 2$ and $\kappa - \frac{1}{\kappa} >0$ for $\kappa> 1$, it is sufficient to prove that 
$$h(\kappa) := \textup{ln}(2) - 2\ln \kappa + (\textup{ln}(2))\left(\kappa - \frac{1}{\kappa}\right) \geq 0, \quad \textup{for all }\kappa \geq 1.$$
Observe that 
$h'(\kappa) = - \frac{2}{\kappa} + (\textup{ln}( 2))\left(1+\frac{1}{\kappa^2}\right) $ and
$h''(\kappa) = \frac{2}{\kappa^3}\,(\kappa - \textup{ln}(2)).$ Since $\kappa \geq 1>\textup{ln}(2)$, we have that $h''(\kappa) >0$ and thus $h(\kappa)$ is convex in the domain $[1, \infty)$. Therefore, $h$ achieves its global minimum at $\kappa^*$ satisfying,
$h'(\kappa^*) = 0$, i.e., $\kappa^*$ satisfies $\kappa^* + \frac{1}{\kappa^*} = \frac{2}{\textup{ln}(2)}.$ The only value of $\kappa^*$ greater than $1$ is $\frac{1}{\textup{ln}(2)} + \sqrt{(\frac{1}{\textup{ln}(2)})^2 - 1}.$ Plugging this value into $h$ gives a value of $\textup{ln}(2)  - 2\textup{ln}\left(\frac{1}{\textup{ln}(2)} + \sqrt{(\frac{1}{\textup{ln}(2)})^2 - 1}\right) + \textup{ln}(2)(\sqrt{ (\frac{2}{\textup{ln}(2)})^2 - 4}) \approx 0.3161 > 0$. 
%Letting $l:= \textup{ln}(2)$ we have $v = l - 2\textup{ln}\left(\frac{1 + \sqrt{1 - l^2}}{l}\right) + 2\sqrt{1 - l^2} = l - 2\textup{ln}( 1+ \sqrt{1 - l^2}) + 2\textup{ln}(l) + 2\sqrt{1 - l^2} \approx 0.3161 > 0$. 
Thus, $h(\kappa) >0$ for all $\kappa > 1$, completing the proof.  
\end{proof}

%Interestingly, the lower bound is not monotone with respect to  dimension~$n$. %\textcolor{red}{Ques:I must be missing something here...as n increases, allowed values of m increase, so it seems the bound monotonically increases with n...}

\begin{proposition}\label{prop:dbound-stSimplex-lb}
Let    {$n \geq 2$},  $\cG=\Delta^n$ and $\kappa>1$. %{\color{red}such that $n^{1-\kappa} + n^{1-1/\kappa}\ge2$}. 
Then, we have
$$D_{\cS^\kappa, \conv(\cS^\kappa)}  := \max_{ (\hat x, \hat y)\in  \conv(\cS^\kappa)  } \min_{ (x,y)\in \cS^\kappa} \| (x-\hat x, y-\hat y) \|_1 \ge 1-n^{1-\kappa}  .$$ 
\end{proposition}
 \begin{proof}
 {We just fix 
 $\hat x_i = \hat y_i = \frac{1}{n}$ for the outer optimization problem. We note that the following is a lower bound:
 \begin{eqnarray*}
    \min_{(x,y) \in \mathcal{S}^\kappa}  \sum_{i=1}^n \left|x_i - \frac{1}{n}\right| + \sum_{i=1}^m \left | y_i - \frac{1}{n}\right |  
     &=& \min_{ x  \in \Delta^n} \sum_{i=1}^n \left|x_i - \frac{1}{n}\right|  + \sum_{i=1}^n \left|x_i^\kappa - \frac{1}{n}\right| \\
     &\geq & \min_{ x  \in \mathbb{R}^n_+} \sum_{i=1}^n \left|x_i - \frac{1}{n}\right|  + \sum_{i=1}^n \left|x_i^\kappa - \frac{1}{n}\right| + \sum_{i =1}^n x_i - 1, 
 \end{eqnarray*}
 where the last inequality follows by taking a particular Lagrangian relaxation.
 Due to {Lemma~\ref{lem:tech}}, we know that the minimizer of the single variable function
 $$f(x_i) = \sum_{i=1}^n \left|x_i - \frac{1}{n}\right|  + \sum_{i=1}^n \left|x_i^\kappa - \frac{1}{n}\right| + x_i $$
 is achieved at $x_i = \frac{1}{n}$. Thus, we obtain
 $$\min_{ x  \in \mathbb{R}^n_+} \sum_{i=1}^n \left|x_i - \frac{1}{n}\right|  + \sum_{i=1}^n \left|x_i^\kappa - \frac{1}{n}\right| + \sum_{i =1}^n x_i - 1  = n\left(\frac{1}{n} - \left(\frac{1}{n}\right)^{\kappa}\right) = 1 - n^{1- \kappa},$$
 completing the proof.}
 \end{proof}

Figure~\ref{fig:dbound} illustrates the comparison of bounds derived in Propositions~\ref{prop:dbound-stSimplex-ub} and \ref{prop:dbound-stSimplex-lb}. As expected,  bounds converge to 1 as $n$ increases and the convergence is faster with larger $\kappa$. % On the other hand, the lower bound for $\kappa=1.25$ attains its maximum value at $m=5$ whereas the lower bounds for $\kappa=2$ and $\kappa=3$ attain their maximum values at $m=3$.  

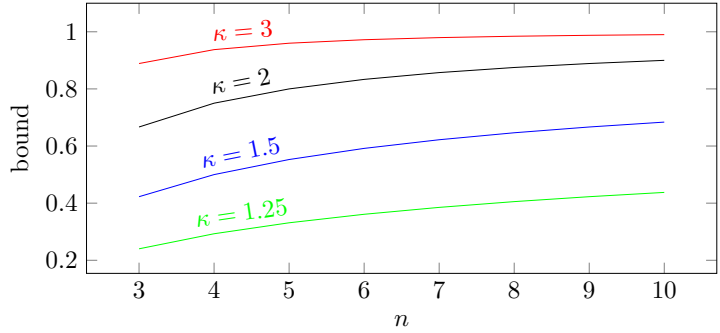
\begin{figure}[H]
%\caption{Upper bound derived in Proposition~~\ref{prop:dbound-stSimplex-ub} (solid line) vs. lower bound derived in Proposition~\ref{prop:dbound-stSimplex-lb} (dashed line) for different $\kappa$ and $n$ values.
\caption{Bounds derived in Propositions~\ref{prop:dbound-stSimplex-ub} and~\ref{prop:dbound-stSimplex-lb}  for different $\kappa$ and $n$ values.
}
\label{fig:dbound}
\centering
%%%
\pgfplotstableread[col sep=&, header=true]{
n	&	ub125	&	lb125	&	ub2	&	lb2	&	ub3	&	lb3	&ub15		&lb15	
%2	&	0.159103585	&	0.112503224	&	0.5	&	0.353553391	&	0.75	&	0.530330086	0.292893219	&	0.292893219	&
3	&	0.240164314	&	0.138658932	&	0.666666667	&	0.384900179	&	0.888888889	&	0.530330086	0.422649731	&	0.422649731	&
4	&	0.292893219	&	0.146446609	&	0.75	&	0.384900179	&	0.9375	&	0.530330086	0.5	&	0.5	&
5	&	0.331259695	&	0.148143839	&	0.8	&	0.384900179	&	0.96	&	0.530330086	0.552786405	&	0.552786405	&
6	&	0.361056896	&	0.148143839	&	0.833333333	&	0.384900179	&	0.972222222	&	0.530330086	0.59175171	&	0.59175171	&
7	&	0.385211847	&	0.148143839	&	0.857142857	&	0.384900179	&	0.979591837	&	0.530330086	0.622035527	&	0.622035527	&
8	&	0.405396442	&	0.148143839	&	0.875	&	0.384900179	&	0.984375	&	0.530330086	0.646446609	&	0.646446609	&
9	&	0.422649731	&	0.148143839	&	0.888888889	&	0.384900179	&	0.987654321	&	0.530330086	0.666666667	&	0.666666667	&
10	&	0.437658675	&	0.148143839	&	0.9	&	0.384900179	&	0.99	&	0.530330086	0.683772234	&	0.683772234	&
}\mydata
     \begin{tikzpicture}[scale=0.95]
\begin{axis}[
   xtick=data,
      ytick={0.0,0.2,0.4,0.6,0.8,1.0},
   xticklabels from table={\mydata}{n},      
   			    xlabel={$n$},
                ylabel={bound},
   ymax=1.1,  
     y=40mm, x=10.5mm,
  ]
\addplot[color=green]   table [ y=ub125, x expr=\coordindex,] {\mydata} node[pos=0.2, sloped, above, text=green] {$\kappa=1.25$}; 
%\addplot[color=green,dashed]   table [ y=lb125, x expr=\coordindex,] {\mydata} node[pos=0.2, sloped, below, text=green] {$\kappa=1.25$};  

\addplot[color=blue]   table [ y=ub15, x expr=\coordindex,] {\mydata} node[pos=0.2, sloped, above, text=blue] {$\kappa=1.5$}; 
%\addplot[color=green,dashed]   table [ y=lb125, x expr=\coordindex,] {\mydata} node[pos=0.2, sloped, below, text=green] {$\kappa=1.25$};  

\addplot[color=black]   table [ y=ub2, x expr=\coordindex,] {\mydata} node[pos=0.2, sloped, above, text=black] {$\kappa=2$}; 
%\addplot[color=black, dashed]   table [ y=lb2, x expr=\coordindex,] {\mydata} node[pos=0.2, sloped, above, text=black] {$\kappa=2$};  

\addplot[color=red]   table [ y=ub3, x expr=\coordindex,] {\mydata} node[pos=0.2, sloped, above, text=red] {$\kappa=3$}; 
%\addplot[color=red, dashed]   table [ y=lb3, x expr=\coordindex,] {\mydata} node[pos=0.2, sloped, above, text=red] {$\kappa=3$};  

\end{axis}
    \end{tikzpicture}
    \end{figure}

\subsubsection{Objective function-based approximation results}
\label{sec:objFunc-approx}

Our first objective function-based approximation result below characterizes the cases under which the \textbf{P} relaxation is exact.   

\begin{proposition}\label{prop:stSimplex-whenPOWexact}
    Let $\cG=\Delta^n$ and $\kappa>1$. Then, we have $O_{\cS^\kappa,\cS^\kappa_P} (\alpha,\beta) =0$ under any of the following cases:
    \begin{itemize}
        \item Case 1: $ J^+:= \{j: \beta_j > 0 \} = \{1,\dots,n\} $.
        \item Case 2: $ J^-:=\{j: \beta_j \le 0 \} = \{1,\dots,n\} $.
        \item Case 3: $ J^+ \neq \emptyset$, $J^- \neq \emptyset$ and either of the following holds, where $\hat J := \{j\in J^+: \alpha_j - \mu < 0\} $ with $\mu:= \min\{ (\alpha_j+\beta_j) : j \in J^- \} $.
            \begin{itemize}
                \item Subcase 3a: $\hat J = \emptyset$.
                \item Subcase 3b: $\hat J \neq \emptyset$ and $\sum_{j\in \hat J} \left(\frac{\mu-\alpha_j}{\kappa\beta_j}\right)^{\frac{1}{\kappa-1}} \ge 1$.
            \end{itemize}
    \end{itemize}
\end{proposition}
\begin{proof}
Consider     problems $z(\alpha, \beta):=     \min\left\{  \alpha^T x + \beta^T y : (x,y) \in \cS^\kappa \right\} $ and 
\begin{equation}\label{eq:linearObjS0}
z_P(\alpha, \beta):=     \min\left\{  \alpha^T x + \beta^T y : (x,y) \in \cS^\kappa_P \right\} .
\end{equation}
% and
% \begin{equation}\label{eq:linearObjS}
% z(\alpha, \beta):=     \min\left\{  \alpha^T x + \beta^T y : (x,y) \in \cS^\kappa \right\} . 
% \end{equation}
Let us split the indices as $J^+ := \{j: \beta_j > 0 \}$ and $J^- := \{j: \beta_j \le 0 \}$. 
Observe that {when optimizing over $\cS^\kappa_P$}, there exists an optimal solution with the following structure (see Proposition~\ref{prop:optimizeOverPOW}): $y_j= x_j^\kappa$ for $j\in J^+$, and  $y_j= x_j$ for $j\in J^-$. Therefore, the value of $z_P(\alpha, \beta)$ defined in~\eqref{eq:linearObjS0}    is equivalent to the following:
\begin{equation*}   \min\left\{ \sum_{j\in J^+} (\alpha_j x_j + \beta_j x_j^\kappa) + \sum_{j\in J^-} (\alpha_j + \beta_j) x_j  : \sum_{j\in J^+} x_j + \sum_{j\in J^-} x_j = 1, x \ge 0 \right\}. 
\end{equation*}

\begin{itemize}
    \item Case 1: $ J^+ = \{1,\dots,n\} $. In this case, problem~\eqref{eq:linearObjS0} is exact since $y_j = x_j^\kappa$ for each $j$, that is, we have $z_P(\alpha, \beta) = z(\alpha, \beta)$.
    \item Case 2: $ J^- = \{1,\dots,n\} $. In this case, problem~\eqref{eq:linearObjS0} can be solved as a linear program and the optimal solution is at an extreme point of the simplex. Therefore, we have $x_{j'}=1$ for $j'=\arg\min\{ (\alpha_j+\beta_j) \}$, and $x_j=0$ for $j\neq j'$. Hence, $y_j=x_j^\kappa$ for each $j$ and the relaxation is exact, that is, we have $z_P(\alpha, \beta) = z(\alpha, \beta)$.
    \item Case 3: $ J^+ \neq \emptyset$ and $J^- \neq \emptyset$. Let $j'=\arg\min\{ (\alpha_j+\beta_j) : j \in J^- \}$ and set $\mu = \min\{ (\alpha_j+\beta_j) : j \in J^- \} $.
    Notice that we must have $x_{j'} = 1-\sum_{j\in J^+}$ and $x_j=0$ for $j\in J^- \setminus\{j'\}$ in an optimal solution. With this simplification, we obtain the following 
     equivalent problem for the value of $z_P(\alpha, \beta) $:
    \begin{equation*} 
    \begin{split}
    &    \min\left\{ \sum_{j\in J^+} (\alpha_j x_j + \beta_j x_j^\kappa) +  \mu (1-\sum_{j\in J^+} x_j)   : \sum_{j\in J^+} x_j \le 1, x_j \ge 0 , j\in J^+ \right\} \\
    =
   & \mu +  \min\left\{ \sum_{j\in J^+} [ (\alpha_j-\mu) x_j + \beta_j x_j^\kappa]     : \sum_{j\in J^+} x_j \le 1, x_j \ge 0 , j\in J^+ \right\} \\
    =
   &  \mu + \min\left\{ \sum_{j\in \hat J} [ (\alpha_j-\mu) x_j + \beta_j x_j^\kappa]     : \sum_{j\in \hat J} x_j \le 1, x_j \ge 0 , j\in \hat J \right\}, 
        \end{split}
    \end{equation*}
    where $\hat J := \{j\in J^+: \alpha_j - \mu < 0\}$ (note that, in an optimal solution, $x_j = 0$ for $j\in J^+ \setminus \hat J$).  
    In this case, the relaxation is exact if $\sum_{j\in \hat J}  x_j \in \{0,1\}$.
    Let us look at subcases:
    \begin{itemize}
        \item Subcase 3a: $\hat J = \emptyset$. In this case, the relaxation is exact as $x_{j'}=y_{j'}=1$ and $x_j=y_j=0$ for $j\neq j'$, that is, we have $z_P(\alpha, \beta) = z(\alpha, \beta)$.
        \item Subcase 3b: $\hat J \neq \emptyset$ and $\sum_{j\in \hat J}  \left( \frac{\mu-\alpha_j}{\kappa\beta_j} \right)^{\frac{1}{\kappa-1}} \ge 1$.   In this case, the relaxation is exact as $\sum_{j \in \hat J} x_{j}=1$.
        In this case, although we do not have a closed form expression of $z_P(\alpha, \beta)$, we know that  $z_P(\alpha, \beta) = z(\alpha, \beta)$.
    \end{itemize}
\end{itemize}
    
\end{proof}
     Notice that when  $ J^+ \neq \emptyset$ and $J^- \neq \emptyset$, $\hat J \neq \emptyset$ and $\sum_{j\in \hat J}  \left( \frac{\mu-\alpha_j}{\kappa\beta_j} \right)^{\frac{1}{\kappa-1}} < 1$, the \textbf{P} relaxation is inexact as $\sum_{j \in \hat J} x_{j} \in (0, 1)$, that is, we have
        $z_P(\alpha, \beta) < z(\alpha, \beta)$. However, note that in this case we know the value of $z_P(\alpha, \beta) = \mu +  (1-\kappa)\sum_{j\in \hat{J}}\beta_j  \left( \frac{\mu-\alpha_j}{\kappa\beta_j} \right)^{\frac{\kappa}{\kappa-1}}  $, but not $z(\alpha, \beta)$.

% {\color{blue}omit the proposition below?
% \begin{proposition}\label{prop:conv1}
% Let $T  = \{(\alpha_j, \beta_j)\}_{j = 1}^n\subseteq \mathbb{R}^{2n},$ be the set satisfying (using notation from above) $J^{-}, J^{+} \neq \emptyset$. Then, we have 
% \begin{eqnarray*}
%     \conv(\cS^\kappa) = \cS^\kappa_P \cap \bigcap_{(\alpha, \beta) \in T}\left\{ \alpha^T x + \beta^T y  \geq z(\alpha, \beta) \right\}.
% \end{eqnarray*}
% % \begin{proof}
% %     Note that in case 1, and case 2 (in above analysis), optimizing a linear objective is exact over $\mathcal{S}$ and $\mathcal{S_0}$. So we only consider case 3.
% % \end{proof}
% \end{proposition}
% }

 Proposition~\ref{prop:stSimplex-whenPOWexact} allows us to estimate the probability that the \textbf{P} relaxation is exact when $\cG=\Delta^n$. For this purpose, we generate $10^{6}$ random $(\alpha,\beta)$ vectors and count the number of times the condition $O_{\cS^\kappa,\cS^\kappa_P} (\alpha,\beta) =0$ is satisfied. We report the results of this simulation in Tables~\ref{tab:simulateUniform} and \ref{tab:simulateNormal} for two different distributions.

\begin{table}[H]
    \centering\small
% Table generated by Excel2LaTeX from sheet '100000'
\begin{tabular}{c|ccccccccc}

 $\kappa \backslash n$          &          2 &          4 &          8 &          16 &          32 &          64   \\
\hline
 
      1.25 &     0.9154 &     0.8696 &     0.8937 &     0.9550 &     0.9940 &     0.9999   \\

       1.5 &     0.9064 &     0.8607 &     0.8896 &     0.9540 &     0.9940 &     0.9999  \\

      1.75 &     0.8990 &     0.8551 &     0.8877 &     0.9542 &     0.9939 &     0.9999   \\

         2 &     0.8933 &     0.8518 &     0.8882 &     0.9540 &     0.9940 &     0.9999   \\

       2.5 &     0.8857 &     0.8490 &     0.8894 &     0.9554 &     0.9942 &     0.9999   \\

         3 &     0.8804 &     0.8508 &     0.8925 &     0.9571 &     0.9943 &     0.9999   \\

\end{tabular}  
    \caption{The estimated probability of the \textbf{P} relaxation being exact via simulation when $\cG=\Delta^n$ and the objective function coefficients $(\alpha,\beta)$ are iid sampled from the $\text{Unif}(-1,1)$ distribution.}
    \label{tab:simulateUniform}
\end{table}

\begin{table}[H]
    \centering\small
% Table generated by Excel2LaTeX from sheet '100000'
\begin{tabular}{c|ccccccccc}

 $\kappa \backslash n$          &          2 &          4 &          8 &          16 &          32 &          64 &          128 &          256 &         512  \\
\hline
        1.25 &     0.9212 &     0.8792 &     0.8788 &     0.9021 &     0.9306 &     0.9551 &     0.9726 &     0.9843 &     0.9913  \\

       1.5 &     0.9123 &     0.8698 &     0.8712 &     0.8960 &     0.9268 &     0.9532 &     0.9721 &     0.9840 &     0.9912  \\

      1.75 &     0.9055 &     0.8630 &     0.8662 &     0.8929 &     0.9254 &     0.9522 &     0.9712 &     0.9835 &     0.9907  \\

         2 &     0.9001 &     0.8577 &     0.8634 &     0.8910 &     0.9242 &     0.9512 &     0.9708 &     0.9833 &     0.9907  \\

       2.5 &     0.8911 &     0.8539 &     0.8611 &     0.8890 &     0.9220 &     0.9504 &     0.9703 &     0.9829 &     0.9906  \\

         3 &     0.8848 &     0.8519 &     0.8604 &     0.8896 &     0.9214 &     0.9497 &     0.9697 &     0.9826 &     0.9905  \\

\end{tabular}  
    \caption{The estimated probability of the \textbf{P} relaxation being exact via simulation when $\cG=\Delta^n$ and the objective function coefficients $(\alpha,\beta)$ are iid sampled from the standard normal distribution.}
    \label{tab:simulateNormal}
\end{table}

Interestingly, the probability of   Case 3a, which is independent of the value~$\kappa$, is quite large in the simulation results and it seems to converge to 1 as $n$ increases. This has motivated us to explore whether  this probability can be computed analytically and shown to converge to 1. Below, we carry out this analysis for the uniform distribution.

\begin{proposition}\label{prop:POWcase3aUniform}
Suppose that $(\alpha,\beta)$ are iid   $\text{Unif}(-1,1)$ distributed.  Then,    \begin{equation*}
\begin{split}
\Pr(\textup{Case 3a}) = 2^{-n}\sum_{m=1}^{n-1} {n \choose m} \left( 1 - m 2^{-m}  \left( I_m + \frac{2^{2(m-n)}}{2n-m} \right) \right) 
,
\end{split}
\end{equation*}
where $I_m := \int_{-1}^0 \left[\frac14 - \frac12z\right]^{n-m}  (1-z)^{m-1} dz $ for $m=1,\dots,n-1$.
\end{proposition}
\begin{proof}
%Suppose that $(\alpha,\beta)$ are iid sampled from the $\text{Unif}(-1,1)$ distribution. 
Let $\cE_m$ be the event in which  $|J^+|=m$ and  $|J^-|=n-m$, $m=0,\dots,n$. Note that $\Pr(\cE_m) = {n \choose m} 2^{-n}$.  Then, Case 1 and Case 2 are respectively the events $\cE_n$ and  $\cE_0$, which happen with probability $2^{-n}$ each. Now, let us understand the event $\cE_m$, $m=1,\dots,n-1$. 

Let $F(z):=\Pr(\alpha_j+\beta_j \ge z| \beta_j \le 0)$. Using Bayes's rule and the fact that the density of the $\text{Unif}(-1,1)$ distribution is $\frac12$ over $[-1,1]$, we deduce that
\[
F(z)=\begin{cases}
    1-\frac14(z+2)^2 & z\in[-2,-1] \\
    \frac14-\frac12z & z\in[-1,0] \\
    \frac14(z-1)^2 & z\in[0,1]
\end{cases}.
\]
Now, let us derive the right-tail probability of $\mu=\min\{\alpha_j+\beta_j : {j\in J^-}\}$ given the event $\cE_m$. Due to the independence of the random variables involved, we have $$G_m(z):=\Pr(\mu 
\ge z | \cE_m) = [F(z)]^{n-m},$$

Next, we derive the CDF of $\nu:=\min\{\alpha_j : j\in J^+\}$ given the event $\cE_m$. In fact, %we have
\[
H_m(z) := \Pr(\nu \le z | \cE_m) = 1 - \Pr(\nu \ge z | \cE_m) = 1 - \left( \frac{1-z}{2}\right)^m=1-2^{-m}(1-z)^m, 
\]
from which we obtain the PDF of $\nu$ given the event $\cE_m$ as
\(
h_m(z) =  m 2^{-m}(1-z)^{m-1} \). 

We are now ready to compute the probability of Case 3a given the event $\cE_m$ as 
\begin{equation}\label{eq:case3a-given-eventm}
\begin{split}
\Pr(\text{Case 3a} | \cE_m)&= \Pr( \nu \ge \mu |\cE_m) \\ &= 1  - \Pr( \mu \ge \nu |\cE_m) = 1-\int_{-1}^1 G_m(z) h_m(z) dz \\
 &= 1-\int_{-1}^0 G_m(z) h_m(z) dz -\int_{0}^1 G_m(z) h_m(z) dz .
\end{split}
\end{equation}
Note that the second integral in the last line of equation~\eqref{eq:case3a-given-eventm} is easy to compute as %we have
\begin{equation*}
\begin{split}
\int_{0}^1 G_m(z) h_m(z) dz &= \int_{0}^1 \left[\frac14(z-1)^2\right]^{n-m} m2^{-m} (1-z)^{m-1} dz \\
&= m 2^{m-2n} \int_{0}^1   (1-z)^{2n-m-1}  dz    \\
 &= m 2^{m-2n}  \left [ -\frac{(1-z)^{2n-m}}{2n-m} \right ]_{z=0}^1  = \frac{ m 2^{m-2n}  }{2n-m}.
\end{split}
\end{equation*}
The computation of the first integral in the last line of equation~\eqref{eq:case3a-given-eventm} is more convoluted. Observe that 
 \begin{equation}\label{eq:POWcase3a-uniform-Im}
 \begin{split}
     \int_{-1}^0 G_m(z) h_m(z) dz &= \int_{-1}^0 \left[\frac14 - \frac12z\right]^{n-m} m2^{-m} (1-z)^{m-1} dz  \\ 
     &= m2^{-m} \int_{-1}^0 \left[\frac14 - \frac12z\right]^{n-m}  (1-z)^{m-1} dz = m2^{-m} I_m.
  \end{split}
\end{equation}  
%and define $I_m := \int_{-1}^0 \left[\frac14 - \frac12z\right]^{n-m}  (1-z)^{m-1} dz $ for $m=1,\dots,n-1$.
It is straightforward to obtain that 
$I_1 =  \int_{-1}^0 \left[\frac14 - \frac12z\right]^{n-1}    dz =
\left[ (-2) \frac{(\frac14-\frac12z)^{n}}{n} \right]_{z=-1}^0   
 = (-2)\frac{(\frac14)^{n}-(\frac34)^{n}}{n}.
$
Now, if $m\ge2$, using integration by parts, we obtain the relation
\begin{equation}\label{eq:Im-recursion-uniform}
\begin{split}
I_m &= \left[ (-2)(1-z)^{m-1}\frac{(\frac14-\frac12z)^{n-m+1}}{n-m+1} \right]_{z=-1}^0 \\ 
& \quad - \frac{2(m-1)}{n-m+1} \int_{-1}^0 \left[\frac14 - \frac12z\right]^{n-m+1}  (1-z)^{m-2} dz  \\
& = (-2)\frac{(\frac14)^{n-m+1}-2^{m-1}(\frac34)^{n-m+1}}{n-m+1} - \frac{2(m-1)}{n-m+1} I_{m-1}.
\end{split}
\end{equation}
Since we have already computed $I_1$, any  $I_m$ with $m=2,\dots,n-1$ can be obtained recursively using the relation above. 

Finally, we conclude that the probability of Case 3a is computed as
\begin{equation*}
\begin{split}
\Pr(\text{Case 3a}) &= \sum_{m=1}^{n-1} \Pr(\text{Case 3a} | \cE_m) \Pr(\cE_m) = 
2^{-n}\sum_{m=1}^{n-1} {n \choose m}\Pr(\text{Case 3a} | \cE_m)\\ & = 
2^{-n}\sum_{m=1}^{n-1} {n \choose m} \left( 1 - m 2^{-m}  \left( I_m + \frac{2^{2(m-n)}}{2n-m} \right) \right) .
\end{split}
\end{equation*}
%from which the result follows.
\end{proof}

The above analysis enables us to compute the \textit{exact} probability of the \textbf{P} relaxation being exact when $\cG=\Delta^n$ and the objective function coefficients $(\alpha,\beta)$ are iid $\text{Unif}(-1,1)$ distributed. These probabilities are reported in Table~\ref{tab:exact3aUniform}.
\begin{table}[H]
    \centering\small
% Table generated by Excel2LaTeX from sheet '100000'
\begin{tabular}{c|cccccccc}

 $ n$          &          2 &          4 &          8 &          16 &          32 &          64     \\
\hline
 
$\Pr(\text{Case 3a})$  &       0.3541 &     0.6434 &     0.8242 &     0.9409 &     0.9930 &     0.9999  \\

\end{tabular}  
    \caption{The probability of the \textbf{P} relaxation being exact when $\cG=\Delta^n$ and the objective function coefficients $(\alpha,\beta)$ are iid $\text{Unif}(-1,1)$ distributed.}
    \label{tab:exact3aUniform}
\end{table}

Now, we are ready to prove that the \textbf{P} relaxation is exact with high probability  for large $n$ when the objective function coefficients are uniformly distributed. 
\begin{theorem}\label{thm:POWexactUniform}
    Suppose that $(\alpha,\beta)$ are iid   $\textup{Unif}(-1,1)$ distributed. 
    Then, 
    $$\Pr(\textup{Case 3a}) \to 1 \text{ as } n \to \infty.$$ 
    {Hence,  $\Pr(O_{\cS^\kappa,\cS^\kappa_P} (\alpha,\beta) =0) \to 1$ as $n\to\infty$.}
\end{theorem}
\begin{proof}
In virtue of Proposition~\ref{prop:POWcase3aUniform}, it suffices to show the following:
    \begin{enumerate}[label=(\roman*)]
        \item $2^{-n}\sum_{m=1}^{n-1} {n \choose m}  \to 1$ as $n \to \infty$.
        \item $2^{-n}\sum_{m=1}^{n-1} {n \choose m}  m 2^{-m}    I_m \to 0 $ as $n \to \infty$.
        \item $2^{-n}\sum_{m=1}^{n-1} {n \choose m}   m 2^{-m}    \frac{2^{2(m-n)}}{2n-m}   \to 0 $ as $n \to \infty$. 
    \end{enumerate}
    In the proof, we  use the binomial identity $(x+y)^n=\sum_{j=0}^n {n \choose j}x^j y^{n-j}$ repeatedly.
    
    To prove (i), observe that
    \[
    2^{-n}\sum_{m=1}^{n-1} {n \choose m}  = 
    2^{-n}\left(\sum_{m=0}^{n} {n \choose m} - {n \choose 0} - {n \choose n} \right) = 1 - 2\times2^{-n} \to 1
    \]
    as $n\to\infty$.

    To prove (ii), we first note that $I_m\ge0$ as $m2^{-m}I_m$ is a probability (see equation~\eqref{eq:POWcase3a-uniform-Im}). From equation~\eqref{eq:Im-recursion-uniform}, we obtain
    \( I_m \le \frac{2^m(\frac34)^{n-m+1}}{n-m+1}  \). Then, we have
    \begin{equation*}
        \begin{split}
 2^{-n}\sum_{m=1}^{n-1} {n \choose m}  m 2^{-m} I_m \le &
 2^{-n}\sum_{m=1}^{n-1} \frac{n!}{(n-m)!m!}  m 2^{-m} \frac{2^m(\frac34)^{n-m+1}}{n-m+1} 
  \\
= & \left(\frac38\right)^{n} \ \sum_{m=1}^{n-1} \frac{n!}{(n-m+1)!(m-1)!}       \left(\frac43\right)^{m-1} \\
= &  \left(\frac38\right)^{n} \ \sum_{m=1}^{n-1} {n \choose m-1} \left(\frac43\right)^{m-1} 1^{n-m+1}
\\
\le & \left(\frac38\right)^{n}  \left(\frac43+1\right)^{n} = \left(\frac38\right)^{n}  \left(\frac73\right)^{n} = \left(\frac78\right)^{n} \to 0 
        \end{split}
    \end{equation*} as $n \to \infty$ (here,  the second inequality follows from the binomial identity).

To prove (iii), observe that
 \begin{equation*}
        \begin{split}
& 2^{-n}\sum_{m=1}^{n-1} {n \choose m}   m 2^{-m}    \frac{2^{2(m-n)}}{2n-m}   = 
2^{-3n}\sum_{m=1}^{n-1} {n \choose m}    2^{m}    \frac{m}{2n-m}  \\
& \le 2^{-3n}\sum_{m=1}^{n-1} {n \choose m}    2^{m}   = 2^{-3n}\sum_{m=1}^{n-1} {n \choose m}    2^{m}   1^{n-m} \le 2^{-3n} 3^n = 
\left(\frac38\right)^{n}  \to 0 
        \end{split}
    \end{equation*} as $n \to \infty$ (here, the first inequality follows since $\frac{m}{2n-m}\le1$ and the second inequality follows from the binomial identity).
\end{proof}

Our last objective function-based approximation result below gives an upper bound on the largest difference between the optimal objective function values when the same linear function is minimized over the nonconvex set $\cS^\kappa$ and the set $\cS_P^\kappa$. % (the feasible region of the \textbf{P} relaxation).   

\begin{proposition}\label{prop:Obound-stSimplex-kappaGeneral}
Let    $\cG=\Delta^n$ and $\kappa>1$. Then, we have
$$O_{\cS^\kappa,\cS^\kappa_P} (\alpha,\beta) \le (-\beta_{j'}) \left(\kappa^{\frac{1}{1-\kappa}}-\kappa^{\frac{\kappa}{1-\kappa}}  \right) , $$
where   $j'=\arg\min\{ (\alpha_j+\beta_j) : \beta_j \le 0 \}$.
\end{proposition}
\begin{proof}
  Due to Proposition~\ref{prop:optimizeOverPOW},  problem $\min_{(x,y)\in \cS^\kappa_P}\{ \alpha^Tx+\beta^Ty\} $ is equivalent~to 
  \begin{equation*}
      \begin{split}
& \min_{x\in\RR^n_+} \left\{ \sum_{j:\beta_j>0} (\alpha_j x_j + \beta_j x_j^\kappa) + \sum_{j:\beta_j\le0} (\alpha_j+ \beta_j)  x_j : \
  \sum_{j:\beta_j>0} x_j +  \sum_{j:\beta_j\le0}  x_j = 1
\right \}  .
      \end{split}
  \end{equation*}
  Since the objective function is linear in $x_j$ variables with $\beta_j\le0$ and they have identical coefficients in the constraint, at most one of them, in particular, the one with the smallest $\alpha_j+\beta_j$ value, can take nonzero value in an optimal solution. Note that the largest possible error is calculated as $\max_{x\in[0,1]} \{ x-x^\kappa\}=\kappa^{\frac{1}{1-\kappa}}-\kappa^{\frac{\kappa}{1-\kappa}}$. Hence, the result follows.
\end{proof}
Notice that $\kappa^{\frac{1}{1-\kappa}}-\kappa^{\frac{\kappa}{1-\kappa}}$ converges to 0 and 1 as $\kappa\to1^+$ and $\kappa\to\infty$, respectively, and it is equal to $\frac14$ for $\kappa=2$.

\subsection{Valid inequalities obtained via RPT}\label{sec:rpt}

In this section, we use RPT~\cite{zhen2021extension,bertsimascone,dey2025second} to derive power cone representable valid inequalities for $\cS^\kappa$.  
Adapting the convention $\frac00=0$,  the specific application of this technique to our setting yields
\begin{eqnarray*}1 &=& \sum_{j=1}^nx_j
\Rightarrow x^{\kappa}_i = \sum_{j=1}^nx_jx^{\kappa}_i 
\Rightarrow x^{\kappa}_i =\sum_{j=1}^n\frac{x^{\kappa}_jx^{\kappa}_i}{x^{\kappa -1}_j} 
\Rightarrow y_i 
= \sum_{j=1}^n\frac{X^{\kappa}_{ij}}{x^{\kappa -1}_j} 
\Rightarrow y_i 
\geq \sum_{j=1}^n\frac{X^{\kappa}_{ij}}{x^{\kappa -1}_j},
\end{eqnarray*}
where we claim that the last inequality is power cone representable for $i=1,\dots,n$.

To formalize the above derivation, let us first define the RPT set  $\cS_V^{\kappa}  = \{ (x,y)\in\Delta^n\times\RR^n : \exists X \in \mathbb{S}^n, w \in \mathbb{R}^{m \times n} :   \eqref{eq:rpt} 
 \}$, where
\begin{equation}\label{eq:rpt}
\begin{split} 
X \ge0, w \ge 0,  \  y_i \ge \sum_{j=1}^n w_{ij}, i = 1,\dots,n , \ 
         (&w_{ij},x_{j},X_{ij})\in\cC^{1/\kappa},  i,j = 1,\dots,n .
\end{split}
\end{equation}
We now show that the RPT set is an outer-approximation of the nonconvex set $\cS^\kappa$, which is given for completeness below.
\begin{proposition}\label{prop:validineqRPT}
For $\cG=\Delta^n$, we have $\cS^\kappa \subseteq \cS^{\kappa}_V$.
\end{proposition}
\begin{proof}
Let $(x,y)\in\cS^\kappa$. For $i,j=1,\dots,n$, define 
$$X_{ij}=x_ix_j \text{ and } w_{ij}= \begin{cases}
\frac{X_{ij}^\kappa}{x_j^{\kappa-1}} & \text{ if }   x_j > 0 \\
0  & \text{ if }   x_j = 0 
\end{cases} .
$$  
We will now show that  constraints in the definition of the set $\cS_V^{\kappa}$ are satisfied. Firstly, the constraint $y_i \ge \sum_{j=1}^n w_{ij}$ is satisfied since we have
\[
y_i = x_i^\kappa = x_i^\kappa \sum_{j=1}^n x_j = x_i^\kappa \sum_{j=1 : x_j>0}^n \frac{x_j^\kappa}{x_j^{\kappa-1}} =  \sum_{j=1 : x_j>0}^n \frac{X_{ij}^\kappa}{x_j^{\kappa-1}} = \sum_{j=1 : x_j>0}^n w_{ij} = \sum_{j=1 }^n w_{ij}.
\]
Secondly, we consider the constraint $(w_{ij},x_{j},X_{ij})\in\cC^{1/\kappa}$, which is trivially satisfied if $x_i=0$ or $x_{j}=0$. Now, let us assume that $x_i>0$ and $x_j>0$, and observe that % we have 
\[
w_{ij}^{1/\kappa} x_j^{1-1/\kappa} = \bigg(\frac{X_{ij}^\kappa}{x_j^{\kappa-1}}\bigg)^{1/\kappa}x_j^{1-1/\kappa} =  \frac{X_{ij}}{x_j^{1-1/\kappa}} x_j^{1-1/\kappa} = X_{ij} .
\]
Hence, we prove that $(x,y) \in \cS^{\kappa}_V$ and the result follows. 
\end{proof}

%\begin{remark}
%*** POW+RPT = POW ***
%\end{remark}
{%\color{red}
\begin{remark}
    Although Proposition~\ref{prop:validineqRPT} is given for $\cG=\Delta^n$, it can easily be extended to a more general ground set $\cG=\{x\in[0,1]^n : Ax = b\}$ with $A \in  \mathbb{R}_+^{m\times n} $ and $b \in \mathbb{R}^m_+$. In this case, the RPT set is obtained as $\cS_V^{\kappa}  = \{ (x,y)\in\cG \times\RR^n : \exists X \in \mathbb{S}^n, w \in \mathbb{R}^{m \times n} :   \eqref{eq:rptGenRmk} 
 \}$, where
\begin{equation}\label{eq:rptGenRmk}
\begin{split} 
X \ge0, w \ge 0,  \ &b_\iota y_i \ge \sum_{j=1}^n A_{\iota j} w_{ij}, \iota=1,\dots,m, i = 1,\dots,n , \\ 
         (&w_{ij},x_{j},X_{ij})\in\cC^{1/\kappa},  i,j = 1,\dots,n .
\end{split}
\end{equation}
\end{remark}

}

\section{Computational Experiments}
\label{sec:computation}

In this section, we present the results of our computational experiments, where the ground set $\cG$ is chosen as the standard simplex~$\Delta^n$.

\subsection{Computational setting}
\label{sec:comp-setting}

We use a 64-bit workstation with two Intel(R) Xeon(R) Gold 6248R CPU (3.00GHz) processors (256 GB  RAM) and the Python programming language in our computational study. We utilize BARON 25.7.29 \cite{baronzhang2024solving}  to solve the nonconvex problems via Pyomo  6.4.0,  and MOSEK 10.0.40 to solve the conic programming relaxations via CVXPY 1.5.2
with the default settings.

% BARON version 25.7.29. Built: WIN-64 2025-07-29 09:36:39

% C:\Users\BK-FENS-3>python -m pip show cvxpy Name: cvxpy Version: 1.5.2

% C:\Users\BK-FENS-3>python -m pip show pyomo Name: Pyomo Version: 6.4.0

% C:\Users\BK-FENS-3>python -m pip show mosek Name: Mosek Version: 10.0.40

We compare the strength of the following nine relaxations against the nonconvex program defined over the set $\cS^\kappa$, which we will refer to as the \textbf{NON} model: 
\begin{itemize}
    \item The \textbf{P} relaxation, defined over  $ \cS^\kappa_P$
    \item The \textbf{PR} relaxation, defined over  $ \cS^\kappa_{P,R}$ %$:= \cS^\kappa_P \cap \cS_R \cap \cS_T^\kappa  $
    \item The \textbf{PRs} relaxation, defined over $ \cS^\kappa_{P,R,s}$%$:= \cS^\kappa_P \cap \cS_R \cap \cS_s \cap \cS_T^\kappa  $
    \item The \textbf{PRs$^3$} relaxation, defined over  
    $ \cS^\kappa_{P,s^3}:= \{ (x,y)\in\cS^\kappa_P : \exists X\in\mathbb{S}^n : 
    \eqref{eq:rlt-stSimplex},
    \eqref{eq:s3by3}, \eqref{eq:transformation}\} $, where   
    \begin{equation}\label{eq:s3by3}
    \begin{split}
       X \ge 0,  \
      \begin{bmatrix}
          X_{ii} & X_{ij} & x_i \\
          X_{ij} & X_{jj} & x_j \\
          x_i & x_j & 1
      \end{bmatrix} \succeq 0
        ,  1\le i < j \le n 
         . 
        \end{split}
    \end{equation}
    Note that this relaxation   only requires a subset of all $3\times3$ principal minors to be positive semidefinite.
    \item The \textbf{PRS} relaxation, defined over $ \cS^\kappa_{P,R,S}$%$:= \cS^\kappa_P \cap \cS_R \cap \cS_S \cap \cS_T^\kappa   $
    \item The \textbf{PRV} relaxation, obtained by adding RPT constraints~\eqref{eq:rpt} to the \textbf{PR} relaxation %: $ \cS^\kappa_{P,R,V}:= \cS^\kappa_P \cap \cS_R \cap \cS_T^\kappa \cup \cS_V^\kappa $
    \item The \textbf{PRsV} relaxation, obtained by adding RPT constraints~\eqref{eq:rpt} to the \textbf{PRs} relaxation%: $ \cS^\kappa_{P,R,s,V}:= \cS^\kappa_P \cap \cS_R \cap \cS_s \cap \cS_T^\kappa \cup \cS_V^\kappa $
    \item The \textbf{PRs$^3$V} relaxation, obtained by adding RPT constraints~\eqref{eq:rpt} to the \textbf{PRs$^3$} relaxation%: $ \cS^\kappa_{P,R,s^3,V}:= \cS^\kappa_P \cap \cS_R \cap \cS_{s^3} \cap \cS_T^\kappa \cup \cS_V^\kappa $
    \item The \textbf{PRSV} relaxation, obtained by adding RPT constraints~\eqref{eq:rpt} to the \textbf{PRS} relaxation%: $ \cS^\kappa_{P,R,S,V}:= \cS^\kappa_P \cap \cS_R \cap \cS_S \cap \cS_T^\kappa \cup \cS_V^\kappa $
\end{itemize}
% Let us call each of the optimization modeled solved as a method and denote the set of methods as $\$

We run an extensive set of experiments parametrized by three key components:
\begin{itemize}
    \item The distribution of objective function coefficients $\cD$: We choose the objective function coefficients randomly with respect to two different distributions: i) $\textup{Unif}(-1, 1)$, ii) standard normal.
    \item The value of exponent $\kappa$: 
    We choose six different $\kappa$ values given by the following set: $\{1.25, 1.5, 1.75, 2, 2.5, 3\}$.
    \item The value of dimension $n$: 
    We choose nine different $n$ values given by the following set: $\{2,3,\dots,10\}$.
\end{itemize}
We will call a triplet of $(\cD,\kappa, n)$ a \textit{setting} (notice that we have $2\times6\times9=108$ settings in total).
For each setting given by the triplet $(\cD,\kappa, n)$, we repeat the experiment 1000 times, and solve the \textbf{NON} model by BARON and the nine relaxations listed above by MOSEK. Therefore, in total, 108000 instances are created  and  each instance is solved ten times. MOSEK has given an \texttt{UNKNOWN} status for four instances for at least one relaxation, which are excluded from the analysis below.

We record the objective function value of 
each optimization problem solved as $z^{\text{\textbf{M}}}_{\cD,\kappa, n,r}$ and the times it takes to solve it as $t^{\text{\textbf{M}}}_{\cD,\kappa, n,r}$, where 
  $r$ stands for the replication index in the setting $(\cD,\kappa, n)$ and \textbf{M} is the model solved (either the \textbf{NON} model or any of the nine relaxations listed above). Then, for each instance and model, we     compute the \textit{absolute dual gap} as $z^{\text{\textbf{NON}}}_{\cD,\kappa, n,r} - z^{\text{\textbf{M}}}_{\cD,\kappa, n,r}$, and define the cumulative absolute gap of a setting $(\cD,\kappa, n)$ for model \textbf{M} as 
\[
\text{Cumulative Gap}^{\text{\textbf{M}}}_{\cD,\kappa, n} = \sum_{r=1}^{1000} \left(z^{\text{\textbf{NON}}}_{\cD,\kappa, n,r} - z^{\text{\textbf{M}}}_{\cD,\kappa, n,r}\right).
\]
Note that, by construction, $\text{Cumulative Gap}^{\text{\textbf{NON}}}_{\cD,\kappa, n}=0$ for every setting.
The cumulative time of a setting $(\cD,\kappa, n)$ for model \textbf{M} is computed as  \[
\text{Cumulative Time}^{\text{\textbf{M}}}_{\cD,\kappa, n} = \sum_{r=1}^{1000} t^{\text{\textbf{M}}}_{\cD,\kappa, n,r}. 
\]
Finally, we will say that a relaxation is \textit{exact} for an instance if the gap is less than $10^{-4}$ and we record the number of times a model \textbf{M} is exact in the setting $(\cD,\kappa, n)$~as
 \[
\text{\#Exact}^{\text{\textbf{M}}}_{\cD,\kappa, n} = \sum_{r=1}^{1000} {\text{\textbf{1}}} \left( z^{\text{\textbf{NON}}}_{\cD,\kappa, n,r} - z^{\text{\textbf{M}}}_{\cD,\kappa, n,r}  \le 10^{-4}\right) , 
\]
where \textbf{1}($\cdot$) is the indicator function. These three metrics, $\text{Cumulative Gap}^{\text{\textbf{M}}}_{\cD,\kappa, n} $, $\text{Cumulative Time}^{\text{\textbf{M}}}_{\cD,\kappa, n} $ and $\text{\#Exact}^{\text{\textbf{M}}}_{\cD,\kappa, n} $, will be our main performance criteria in the following discussion.

\subsection{Results}
\label{sec:comp-results}

We now provide the details of our computational study. We note that each statistic given below is  averaged over the settings considered.

\subsubsection{Aggregate Results}
\label{sec:comp-results-aggregate}
We first provide the most aggregate results when all the 108 settings are considered together in Figure~\ref{fig:gapTime-all}. We observe that the \textbf{P} relaxation, which is the weakest one, can solve 888 instances on the average and the average cumulative  gap is 3.607 while it only takes 34.429 seconds to solve. The addition of the RLT constraints  has the largest marginal effect, increasing the average number of exact instances to 924 and reducing the average cumulative gap to 1.160 in the \textbf{PR} relaxation by the increase of a mere 4 seconds. The \textbf{PRS} relaxation~further increases the average number of exact instances to 941 and reduces the average cumulative gap to 0.672 while only taking 6 seconds longer than the \textbf{PR} relaxation. Interestingly, the average cumulative gaps reported by the  \textbf{PRs} and   \textbf{PRs$^3$}   relaxations  are very close to the \textbf{PRS} relaxation, albeit the computational cost of these weaker relaxations are relatively high.
The addition of the RPT constraints~\eqref{eq:rpt}  further improves the strength of the relaxations. Relaxations \textbf{PRV} and  \textbf{PRSV} are almost indistinguishable as the number of exact instances (959 vs. 960), the gaps (0.405 vs. 0.398) and the times (121.177	vs. 129.492) are all very close to each other.  We again observe that weaker relaxations \textbf{PRsV} and   \textbf{PRs$^3$V} take more time than the \textbf{PRSV} relaxation.

 {The comparison of   \textbf{PRS} and  \textbf{PRV} relaxations leads to interesting observations. As proven in Theorem~\ref{thm:exactness-kappa=2-stSimplexn=2}, both relaxations are exact when $n=2$ and  $\kappa=2$. The \textbf{PRS} relaxation is also exact for $n\ge3$ and $\kappa=2$ in our experiments whereas the average cumulative gap of the  \textbf{PRV} relaxation is 0.023 and the average number of exact instances is 992 across these 16 settings. In the remaining 90 settings where $\kappa\neq2$, the \textbf{PRV} relaxation is more successful with an average cumulative gap of 0.482 and average number of exact instances of 952 whereas   the \textbf{PRS} relaxation   has an average cumulative gap of 0.806 and average number of exact instances of  929. We note that these relaxations are not comparable in general as there exist instances in which the bound given by the   \textbf{PRS} relaxation is stronger than that of the  \textbf{PRV}  relaxation and vice versa.} 
%  is better when $n\ge 3$ and $\kappa=2$  i. Say something like, ``on average the RPT inequality typically appear to significantly dominate the SDP relaxation, although in the case of degree 2, we observe instances where \textbf{PRS} is tight and  \textbf{PRV} is not tight.
% \\
% **
% Out of 108 settings: i) PRS=PRV when $n=2$ and $\kappa=2$ (two settings). ii) PRS is better when $n\ge 3$ and $\kappa=2$ (16 settings, exact: 1000 vs. 992, gap: 0 vs. 0.023 on avg). iii) PRV is better when  $\kappa\neq2$ (90 settings, exact: 952 vs. 929, gap: 0.482 vs. 0.806 on avg)
% **
%

To summarize, we can say that in circumstances in which the time budget is limited,  \textbf{PR}  and  \textbf{PRS} relaxations offer reasonable alternatives to the \textbf{NON} model solved by BARON. When more computational budget is available, the addition of RPT constraints~\eqref{eq:rpt} can further improve the relaxation quality, {with the \textbf{PRV} relaxation in particular being an effective option}. 
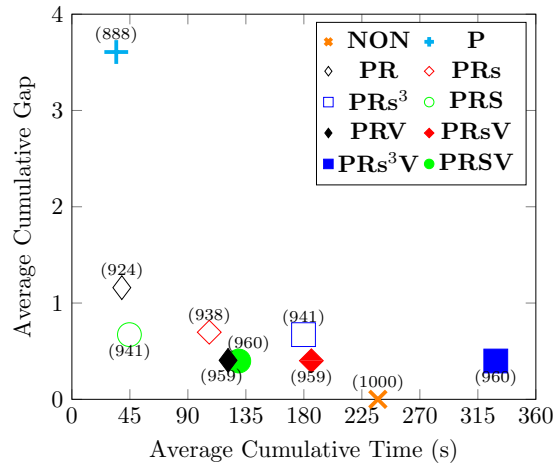
\begin{figure}[H]
	\caption{Average Cumulative Gap vs. Average Cumulative Time with respect to all settings considered. The average number of exact instances is given in parenthesis.}
	\label{fig:gapTime-all}
	\centering
		\begin{tikzpicture}[scale=0.9]
			\begin{axis}[
				xmin=0, 
				xmax=360,
                xtick={0,45,...,360},
				%xmode=log,
				ymin=0, 
				ymax=4.0,
			    xlabel={ Average Cumulative Time (s)}, 
			      ylabel={ Average Cumulative Gap},
				% label style={font=\bfseries\boldmath},
				% tick label style={font=\bfseries\boldmath},
				scatter/classes={
					NON={scale=2.5,mark=x,  color=orange, ultra thick}, 
					POW={scale=2.5,mark=+,  color=cyan, ultra thick}, 
					PR={scale=2.5,mark=diamond,  color=black}, 
					PRSOC={scale=2.5,mark=halfsquare*, color=red, mark color=white, fill=white}, 
					PRS3={scale=2.5,mark=square, color=blue}, 
					PRSDP={scale=2.5,mark=o,  color=green}, 
					PRv={scale=2.5,mark=diamond*,   color=black}, 
					PRSOCv={scale=2.5,mark=halfsquare*, color=red, mark color=red, fill=red}, 
					PRS3v={scale=2.5,mark=square*,   color=blue},
					PRSDPv={scale=2.5,mark=*,  color=green}
				},
           legend entries={\textbf{NON}, \textbf{P}, \textbf{PR}, \textbf{PRs}, \textbf{PRs$^3$}, \textbf{PRS}, \textbf{PRV}, \textbf{PRsV}, \textbf{PRs$^3$V}, \textbf{PRSV}},
                    legend style={
        legend columns=2},
				] 
				\addplot[
				scatter, 
				only marks,
				scatter src=explicit symbolic,
				nodes near coords*={\annotvalue},
				node near coord style={ anchor=\anchorvalue, font=\scriptsize, color=black},
				visualization depends on={value \thisrow{annotation} \as \annotvalue },
				visualization depends on={  value \thisrow{anchor} \as \anchorvalue },
				]
				table[meta=label] {
					x       y            label  annotation  anchor
237.429	0	NON	(1000)	south 
34.429	3.607	POW	(888)	south west
38.751	1.160	PR	(924)	south east
106.599	0.697	PRSOC	(938)	south east
179.725	0.672	PRS3	(941)	south east
44.638	0.672	PRSDP	(941)	north
129.492	0.398	PRSDPv	{ \ \ (960)}	south
121.177	0.405	PRv	{(959) \ \ }	north
185.816	0.401	PRSOCv	(959)	north
328.910	0.398	PRS3v	(960)	north
				};
			\end{axis}
			
		\end{tikzpicture}
    \end{figure}

We again note that the strengths of \textbf{PRs} and \textbf{PRs$^3$} relaxations are very close to the \textbf{PRS} relaxation, and this observation remains valid with the addition of the RPT constraints~\eqref{eq:rpt}. 
Related to the \textbf{PRs$^3$} relaxation,  we have 
the following conjecture:
\begin{cnj}\label{conj:PRS=PRs3}
    We have  $\cS^\kappa_{P,R,s^3} = \cS^\kappa_{P,R,S} $.  
\end{cnj}
Although it might be interesting to pursue   exploring this observation in the future theoretically, from a computational standpoint,  \textbf{PRs}, \textbf{PRs$^3$}, \textbf{PRsV} and \textbf{PRs$^3$V} relaxations are simply too expensive alternatives and dominated by stronger relaxations in terms of  computational effort. Therefore, we omit them from the discussion below.

%
%
%%%%%%%%%%%%%%%%%%%%%%%%%%%%%%%%%%%%%%%%%%%%%%%%%%%%%%%%
%
%

% % Table generated by Excel2LaTeX from sheet 'Sheet1'
% \begin{table}[H]
% 	\centering
% 	\caption{Table of parameter settings for the heuristic approach. The markers are used in Figures~\ref{fig:deltaGap} and \ref{fig:lpGap}.}
% 	\label{tab:param-setting-heuristic}
% % Table generated by Excel2LaTeX from sheet 'Sheet6'
% \begin{tabular}{|c|c|}

%     Method &     Marker \\
% \hline
% {Nonconvex} &  \begin{tikzpicture}[scale=0.01]
% 				\begin{axis}   
% 					[xmin=-1,xmax=1,ymin = -1,ymax = 1,xtick={},ytick={}] 
% 					\addplot[scale=2.5,mark size=500pt, mark=x,   color=orange, ultra thick] coordinates {(0,0)};
% 				\end{axis}
% 		\end{tikzpicture} \\
%  {POW} &            \\

% {POW+RLT} &            \\

% POW+RLT+SOC &            \\

% POW+RLT+SDP3 &            \\

% {POW+RLT+SDP} &            \\

% {POW+RLT+validP} &            \\

% POW+RLT+SOC+validP &            \\

% POW+RLT+SDP3+validP &            \\

% {POW+RLT+SDP+validP} &            \\

% \end{tabular}  
% \end{table}

\newcommand\labelcolor{white}

\subsubsection{The effect of distribution $\cD$}
\label{sec:comp-results-distribution}

We now analyze the effect of the distribution of objective function coefficients $\cD$ on the results, which are reported 
in Figure~\ref{fig:gapTime-objective}. Although  $\text{Unif}(-1,1)$ and standard normal both have mean zero, their variances are 1 and $1/6$, respectively. This is reflected on the outcomes as the average cumulative gap values for the standard normal are higher than those of the uniform distribution. We also note that these observations are consistent with the simulation results that we have conducted in Section~\ref{sec:objFunc-approx} (see Tables~\ref{tab:simulateUniform} and~\ref{tab:simulateNormal}).

\begin{figure}[H]
	\caption{Average Cumulative Gap vs. Average Cumulative Time with respect to different distributions for the objective function coefficients.}
	\label{fig:gapTime-objective}
	\centering
		\begin{subfigure}{.35\textwidth}
		\begin{tikzpicture}[scale=0.53]
			\begin{axis}[
				title={\large{Uniform}},
				xmin=0, 
				xmax=270, xtick={0,45,...,270},
                xtick={0,45,...,270},
				%xmode=log,
				ymin=0, 
				ymax=5,
				    xlabel={\large Average Cumulative Time (s)}, 
				ylabel={\large Average Cumulative Gap},
				% label style={font=\bfseries\boldmath},
				% tick label style={font=\bfseries\boldmath},
				scatter/classes={
					NON={scale=2.5,mark=x,  color=orange, ultra thick}, 
					POW={scale=2.5,mark=+,  color=cyan, ultra thick}, 
					PR={scale=2.5,mark=diamond,  color=black}, 
		%			PRSOC={scale=2.5,mark=halfsquare*, color=red, mark color=white, fill=white}, 
		%			PRS3={scale=2.5,mark=square, color=blue}, 
					PRSDP={scale=2.5,mark=o,  color=green}, 
					PRv={scale=2.5,mark=diamond*,   color=black}, 
		%			PRSOCv={scale=2.5,mark=halfsquare*, color=red, mark color=red, fill=red}, 
		%			PRS3v={scale=2.5,mark=square*,   color=blue},
					PRSDPv={scale=2.5,mark=*,  color=green}
				}
				] 
				\addplot[
				scatter, 
				only marks,
				scatter src=explicit symbolic,
				nodes near coords*={\annotvalue},
				node near coord style={ anchor=\anchorvalue, font=\scriptsize, color=\labelcolor},
				visualization depends on={value (\thisrow{annotation}) \as \annotvalue },
				visualization depends on={  value \thisrow{anchor} \as \anchorvalue },
				]
				table[meta=label] {
					x       y            label  annotation  anchor
237.436	0	NON	1000	south east
34.470	2.522	POW	894	north east
39.011	0.872	PR	927	south east
%105.370	0.492	PRSOC	944	south east
%178.810	0.477	PRS3	946	south east
44.832	0.477	PRSDP	946	north east
131.310	0.275	PRSDPv	964	south east
120.327	0.280	PRv	963	south east
%186.331	0.276	PRSOCv	964	south east
%330.180	0.275	PRS3v	964	south east
				};
			\end{axis}
			
		\end{tikzpicture}
	\end{subfigure}
    \begin{subfigure}{.35\textwidth}
		\begin{tikzpicture}[scale=0.53]
			\begin{axis}[
				title={\large{Normal}},
				xmin=0, 
				xmax=270, xtick={0,45,...,270},
                xtick={0,45,...,270},
				%xmode=log,
				ymin=0, 
				ymax=5,
				   xlabel={\large Average Cumulative Time (s)}, 
				%ylabel={\large Average Cumulative Gap},
				% label style={font=\bfseries\boldmath},
				% tick label style={font=\bfseries\boldmath},
				scatter/classes={
					NON={scale=2.5,mark=x,  color=orange, ultra thick}, 
					POW={scale=2.5,mark=+,  color=cyan, ultra thick}, 
					PR={scale=2.5,mark=diamond,  color=black}, 
				%	PRSOC={scale=2.5,mark=halfsquare*, color=red, mark color=white, fill=white}, 
				%	PRS3={scale=2.5,mark=square, color=blue}, 
					PRSDP={scale=2.5,mark=o,  color=green}, 
					PRv={scale=2.5,mark=diamond*,   color=black}, 
				%	PRSOCv={scale=2.5,mark=halfsquare*, color=red, mark color=red, fill=red}, 
				%	PRS3v={scale=2.5,mark=square*,   color=blue},
					PRSDPv={scale=2.5,mark=*,  color=green}
				},
           legend entries={\textbf{NON}, \textbf{P}, \textbf{PR},  \textbf{PRS}, \textbf{PRV},   \textbf{PRSV}},
                    legend style={
        legend columns=2},
				] 
				\addplot[
				scatter, 
				only marks,
				scatter src=explicit symbolic,
				nodes near coords*={\annotvalue},
				node near coord style={ anchor=\anchorvalue, font=\scriptsize, color=\labelcolor},
				visualization depends on={value (\thisrow{annotation}) \as \annotvalue },
				visualization depends on={  value \thisrow{anchor} \as \anchorvalue },
				]
				table[meta=label] {
					x       y            label  annotation  anchor
237.423	0	NON	1000	south east
34.387	4.692	POW	882	north east
38.491	1.449	PR	920	south east
% 107.827	0.902	PRSOC	933	south east
% 180.641	0.867	PRS3	936	south east
44.445	0.867	PRSDP	936	south east
127.675	0.522	PRSDPv	956	south east
122.028	0.530	PRv	954	south east
% 185.302	0.525	PRSOCv	955	south east
% 327.639	0.522	PRS3v	956	south east
				};
			\end{axis}
			
		\end{tikzpicture}
	\end{subfigure}
    \end{figure}
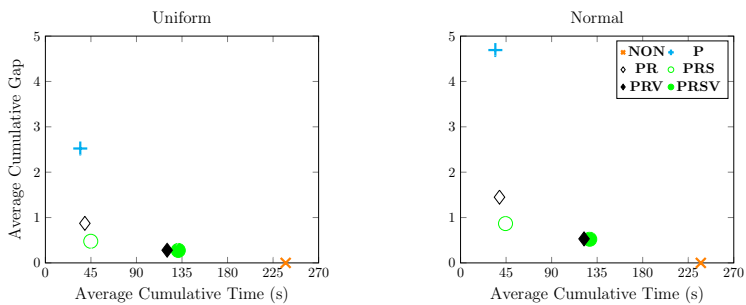

In terms of computational cost, the choice of the distribution does not seem to make a significant difference. % a very similar trend  a significant difference or a consistent trend for different values of $\kappa$.

%
%
%%%%%%%%%%%%%%%%%%%%%%%%%%%%%%%%%%%%%%%%%%%%
%
%

\subsubsection{The effect of exponent $\kappa$}
\label{sec:comp-results-exponent}

We next analyze the effect of  exponent $\kappa$ on the results, which are reported 
in Figure~\ref{fig:gapTime-kappa}. We observe that the average cumulative gap of the  \textbf{P} relaxation increases with $\kappa$, which is an expected outcome due to Proposition~\ref{prop:Obound-stSimplex-kappaGeneral}. Interestingly, the average cumulative gap of other relaxations first increases, then decreases with respect to $\kappa$. {As stated in Conjecture~\ref{conj:exactness-kappa=2-stSimplexn-any}, the \textbf{PRS} relaxation is exact when $\kappa=2$ in all  our experiments.} We also note that all relaxations are quite strong around $\kappa=2$, which suggests that the effect of the RLT constraints is most influential around this value.

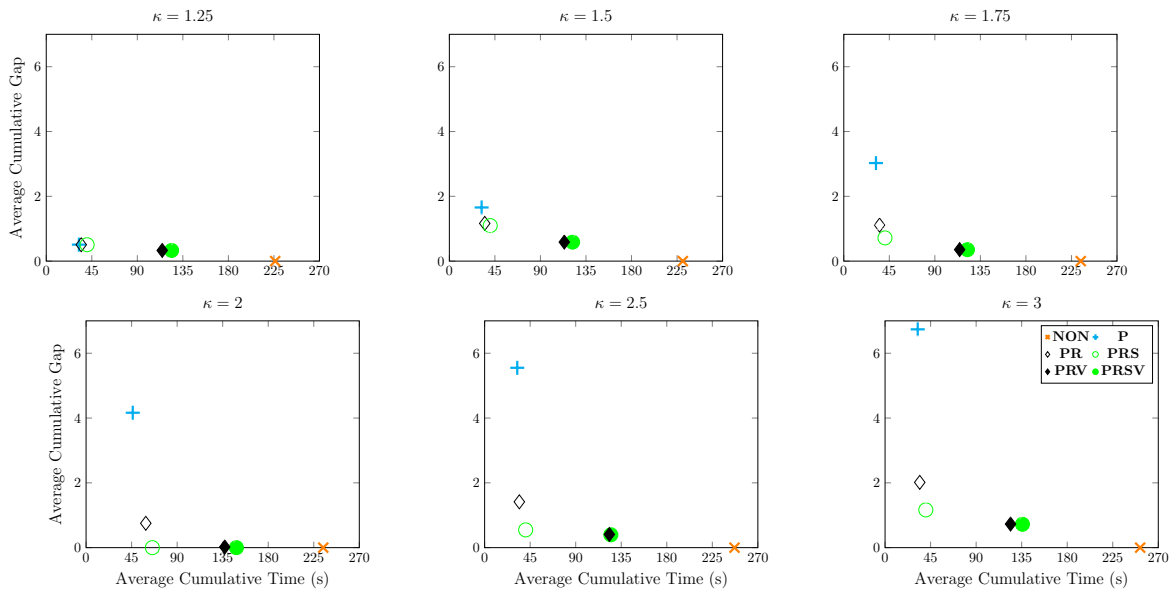
\begin{figure}[H]
	\caption{Average Cumulative Gap vs. Average Cumulative Time with respect to different $\kappa$ values.}
	\label{fig:gapTime-kappa}
	
	\begin{subfigure}{.34\textwidth}
		\begin{tikzpicture}[scale=0.53]
			\begin{axis}[
				title={\large{$\kappa=1.25$}},
				xmin=0, 
				xmax=270, xtick={0,45,...,270}, 
				%xmode=log,
				ymin=0, 
				ymax=7,
				ylabel={\large Average Cumulative Gap},
				% label style={font=\bfseries\boldmath},
				% tick label style={font=\bfseries\boldmath},
				scatter/classes={
					NON={scale=2.5,mark=x,  color=orange, ultra thick}, 
					POW={scale=2.5,mark=+,  color=cyan, ultra thick}, 
					PR={scale=2.5,mark=diamond,  color=black}, 
					%PRSOC={scale=2.5,mark=halfsquare*, color=red, mark color=white, fill=white}, 
					%PRS3={scale=2.5,mark=square, color=blue}, 
					PRSDP={scale=2.5,mark=o,  color=green}, 
					PRv={scale=2.5,mark=diamond*,   color=black}, 
					%PRSOCv={scale=2.5,mark=halfsquare*, color=red, mark color=red, fill=red}, 
					%PRS3v={scale=2.5,mark=square*,   color=blue},
					PRSDPv={scale=2.5,mark=*,  color=green}
				},
				] 
				\addplot[
				scatter, 
				only marks,
				scatter src=explicit symbolic,
				nodes near coords*={\annotvalue},
				node near coord style={ anchor=\anchorvalue, font=\scriptsize, color=\labelcolor},
				visualization depends on={value (\thisrow{annotation}) \as \annotvalue },
				visualization depends on={  value \thisrow{anchor} \as \anchorvalue },
				]
				table[meta=label] {
					x       y            label  annotation  anchor
	226.264	0	NON	1000	south
32.237	0.509	POW	930	south
34.578	0.506	PR	935	south
%108.156	0.506	PRSOC	935	south
%174.763	0.506	PRS3	935	south
40.375	0.506	PRSDP	935	south
123.870	0.330	PRSDPv	954	south
114.604	0.330	PRv	954	south
%187.943	0.330	PRSOCv	954	south
%324.659	0.330	PRS3v	954	south
				};
			\end{axis}
			
		\end{tikzpicture}
	\end{subfigure}
		\begin{subfigure}{.32\textwidth}
		\begin{tikzpicture}[scale=0.53]
			\begin{axis}[
				title={\large{$\kappa=1.5$}},
				xmin=0, 
				xmax=270, xtick={0,45,...,270},
				%xmode=log,
				ymin=0, 
				ymax=7,
                %yticklabel=\empty,
				%    xlabel={\%Gap (delta)}, 
			%	ylabel={\large Average Cumulative Gap},
				% label style={font=\bfseries\boldmath},
				% tick label style={font=\bfseries\boldmath},
				scatter/classes={
					NON={scale=2.5,mark=x,  color=orange, ultra thick}, 
					POW={scale=2.5,mark=+,  color=cyan, ultra thick}, 
					PR={scale=2.5,mark=diamond,  color=black}, 
				%	PRSOC={scale=2.5,mark=halfsquare*, color=red, mark color=white, fill=white}, 
				%	PRS3={scale=2.5,mark=square, color=blue}, 
					PRSDP={scale=2.5,mark=o,  color=green}, 
					PRv={scale=2.5,mark=diamond*,   color=black}, 
				%	PRSOCv={scale=2.5,mark=halfsquare*, color=red, mark color=red, fill=red}, 
				%	PRS3v={scale=2.5,mark=square*,   color=blue},
					PRSDPv={scale=2.5,mark=*,  color=green}
				},
				] 
				\addplot[
				scatter, 
				only marks,
				scatter src=explicit symbolic,
				nodes near coords*={\annotvalue},
				node near coord style={ anchor=\anchorvalue, font=\scriptsize, color=\labelcolor},
				visualization depends on={value (\thisrow{annotation}) \as \annotvalue },
				visualization depends on={  value \thisrow{anchor} \as \anchorvalue },
				]
				table[meta=label] {
					x       y            label  annotation  anchor
230.736	0	NON	1000	south
31.881	1.656	POW	897	north
34.963	1.164	PR	922	north
%103.917	1.102	PRSOC	930	south
%174.849	1.096	PRS3	931	south
40.396	1.096	PRSDP	931	north
121.813	0.587	PRSDPv	942	south
113.594	0.587	PRv	942	south
%187.814	0.587	PRSOCv	942	south
%326.429	0.587	PRS3v	942	south
				};
			\end{axis}
			
		\end{tikzpicture}
	\end{subfigure}
\hfill
    	\begin{subfigure}{.325\textwidth}
		\begin{tikzpicture}[scale=0.53]
			\begin{axis}[
				title={\large{$\kappa=1.75$}},
				xmin=0, 
				xmax=270, xtick={0,45,...,270},
				%xmode=log,
				ymin=0, 
				ymax=7,
                %yticklabel=\empty,
				%    xlabel={\%Gap (delta)}, 
			%	ylabel={\large Average Cumulative Gap},
				% label style={font=\bfseries\boldmath},
				% tick label style={font=\bfseries\boldmath},
				scatter/classes={
					NON={scale=2.5,mark=x,  color=orange, ultra thick}, 
					POW={scale=2.5,mark=+,  color=cyan, ultra thick}, 
					PR={scale=2.5,mark=diamond,  color=black}, 
				%	PRSOC={scale=2.5,mark=halfsquare*, color=red, mark color=white, fill=white}, 
				%	PRS3={scale=2.5,mark=square, color=blue}, 
					PRSDP={scale=2.5,mark=o,  color=green}, 
					PRv={scale=2.5,mark=diamond*,   color=black}, 
				%	PRSOCv={scale=2.5,mark=halfsquare*, color=red, mark color=red, fill=red}, 
				%	PRS3v={scale=2.5,mark=square*,   color=blue},
					PRSDPv={scale=2.5,mark=*,  color=green}
				},
				] 
				\addplot[
				scatter, 
				only marks,
				scatter src=explicit symbolic,
				nodes near coords*={\annotvalue},
				node near coord style={ anchor=\anchorvalue, font=\scriptsize, color=\labelcolor},
				visualization depends on={value (\thisrow{annotation}) \as \annotvalue },
				visualization depends on={  value \thisrow{anchor} \as \anchorvalue },
				]
				table[meta=label] {
					x       y            label  annotation  anchor
234.168	0	NON	1000	south
31.801	3.027	POW	882	south
35.373	1.111	PR	922	south
%104.175	0.743	PRSOC	936	south
%175.365	0.719	PRS3	936	south
40.701	0.717	PRSDP	936	south
122.207	0.353	PRSDPv	945	south
114.509	0.355	PRv	945	south
%188.711	0.354	PRSOCv	945	south
%326.994	0.353	PRS3v	945	south
				};
			\end{axis}
			
		\end{tikzpicture}
	\end{subfigure}
	%%%%%%%%%%%%%%%%%%%%%%%%
	
	\begin{subfigure}{.34\textwidth}
    \centering
		\begin{tikzpicture}[scale=0.53]
			\begin{axis}[
				title={\large{$\kappa=2$}},
				xmin=0, 
				xmax=270, xtick={0,45,...,270},
				%xmode=log,
				ymin=0, 
				ymax=7,
				xlabel={\large Average Cumulative Time (s)}, 
				ylabel={\large Average Cumulative Gap},
                %%yticklabel=\empty,
				% label style={font=\bfseries\boldmath},
				% tick label style={font=\bfseries\boldmath},
				scatter/classes={
					NON={scale=2.5,mark=x,  color=orange, ultra thick}, 
					POW={scale=2.5,mark=+,  color=cyan, ultra thick}, 
					PR={scale=2.5,mark=diamond,  color=black}, 
				%	PRSOC={scale=2.5,mark=halfsquare*, color=red, mark color=white, fill=white}, 
				%	PRS3={scale=2.5,mark=square, color=blue}, 
					PRSDP={scale=2.5,mark=o,  color=green}, 
					PRv={scale=2.5,mark=diamond*,   color=black}, 
				%	PRSOCv={scale=2.5,mark=halfsquare*, color=red, mark color=red, fill=red}, 
				%	PRS3v={scale=2.5,mark=square*,   color=blue},
					PRSDPv={scale=2.5,mark=*,  color=green}
				},
				] 
				\addplot[
				scatter, 
				only marks,
				scatter src=explicit symbolic,
				nodes near coords*={\annotvalue},
				node near coord style={ anchor=\anchorvalue, font=\scriptsize, color=\labelcolor},
				visualization depends on={value (\thisrow{annotation}) \as \annotvalue },
				visualization depends on={  value \thisrow{anchor} \as \anchorvalue },
				]
				table[meta=label] {
					x       y            label  annotation  anchor
234.372	0	NON	1000	south
46.128	4.164	POW	874	south
59.020	0.752	PR	960	south
%107.474	0.028	PRSOC	991	south
%209.304	0.000	PRS3	999	south
65.503	0.000	PRSDP	999	south
148.550	0.000	PRSDPv	999	south
137.054	0.020	PRv	993	south
%182.229	0.006	PRSOCv	997	south
%353.037	0.000	PRS3v	999	south
				};
			\end{axis}
			
		\end{tikzpicture}
	\end{subfigure}
		\begin{subfigure}{.32\textwidth}
            \centering
		\begin{tikzpicture}[scale=0.53]
			\begin{axis}[
				title={\large{$\kappa=2.5$}},
				xmin=0, 
				xmax=270, xtick={0,45,...,270},
				%xmode=log,
				ymin=0, 
				ymax=7,
				xlabel={\large Average Cumulative Time (s)}, 
				%ylabel={\large Average Cumulative Gap},
                %yticklabel=\empty,
				% label style={font=\bfseries\boldmath},
				% tick label style={font=\bfseries\boldmath},
				scatter/classes={
					NON={scale=2.5,mark=x,  color=orange, ultra thick}, 
					POW={scale=2.5,mark=+,  color=cyan, ultra thick}, 
					PR={scale=2.5,mark=diamond,  color=black}, 
				%	PRSOC={scale=2.5,mark=halfsquare*, color=red, mark color=white, fill=white}, 
				%	PRS3={scale=2.5,mark=square, color=blue}, 
					PRSDP={scale=2.5,mark=o,  color=green}, 
					PRv={scale=2.5,mark=diamond*,   color=black}, 
				%	PRSOCv={scale=2.5,mark=halfsquare*, color=red, mark color=red, fill=red}, 
				%	PRS3v={scale=2.5,mark=square*,   color=blue},
					PRSDPv={scale=2.5,mark=*,  color=green}
				},
				] 
				\addplot[
				scatter, 
				only marks,
				scatter src=explicit symbolic,
				nodes near coords*={\annotvalue},
				node near coord style={ anchor=\anchorvalue, font=\scriptsize, color=\labelcolor},
				visualization depends on={value (\thisrow{annotation}) \as \annotvalue },
				visualization depends on={  value \thisrow{anchor} \as \anchorvalue },
				]
				table[meta=label] {
					x       y            label  annotation  anchor
246.926	0	NON	1000	south east
32.249	5.551	POW	871	south east
34.293	1.414	PR	912	south east
%111.117	0.588	PRSOC	929	south east
%172.124	0.548	PRS3	932	south east
40.560	0.548	PRSDP	932	south east
124.734	0.397	PRSDPv	962	south east
123.289	0.408	PRv	961	south east
%185.906	0.403	PRSOCv	962	south east
%320.385	0.397	PRS3v	962	south east
				};
			\end{axis}
			
		\end{tikzpicture}
	\end{subfigure}
\hfill
	\begin{subfigure}{.32\textwidth}
        \centering
		\begin{tikzpicture}[scale=0.53]
			\begin{axis}[
				title={\large{$\kappa=3$}},
				xmin=0, 
				xmax=270, xtick={0,45,...,270},
				%xmode=log,
				ymin=0, 
				ymax=7,
				xlabel={\large Average Cumulative Time (s)}, 
				%ylabel={\large Average Cumulative Gap},
                %yticklabel=\empty,
				% label style={font=\bfseries\boldmath},
				% tick label style={font=\bfseries\boldmath},
				scatter/classes={
					NON={scale=2.5,mark=x,  color=orange, ultra thick}, 
					POW={scale=2.5,mark=+,  color=cyan, ultra thick}, 
					PR={scale=2.5,mark=diamond,  color=black}, 
				%	PRSOC={scale=2.5,mark=halfsquare*, color=red, mark color=white, fill=white}, 
				%	PRS3={scale=2.5,mark=square, color=blue}, 
					PRSDP={scale=2.5,mark=o,  color=green}, 
					PRv={scale=2.5,mark=diamond*,   color=black}, 
				%	PRSOCv={scale=2.5,mark=halfsquare*, color=red, mark color=red, fill=red}, 
				%	PRS3v={scale=2.5,mark=square*,   color=blue},
					PRSDPv={scale=2.5,mark=*,  color=green}
				},
           legend entries={\textbf{NON}, \textbf{P}, \textbf{PR},  \textbf{PRS}, \textbf{PRV},   \textbf{PRSV}},
                    legend style={
        legend columns=2},
				] 
				\addplot[
				scatter, 
				only marks,
				scatter src=explicit symbolic,
				nodes near coords*={\annotvalue},
				node near coord style={ anchor=\anchorvalue, font=\scriptsize, color=\labelcolor},
				visualization depends on={value (\thisrow{annotation}) \as \annotvalue },
				visualization depends on={  value \thisrow{anchor} \as \anchorvalue },
				]
				table[meta=label] {
					x       y            label  annotation  anchor
252.109	0	NON	1000	south
32.277	6.738	POW	871	north
34.278	2.015	PR	892	south
%104.752	1.215	PRSOC	908	south
%171.949	1.166	PRS3	911	south
40.294	1.165	PRSDP	911	south
135.781	0.722	PRSDPv	958	south
124.013	0.728	PRv	957	south
%182.296	0.725	PRSOCv	957	south
%321.953	0.722	PRS3v	958	south
				};
			\end{axis}
			
		\end{tikzpicture}
	\end{subfigure}
    \end{figure}

In terms of computational cost, we do not observe a significant difference or a consistent trend for different values of $\kappa$.
%
%
%%%%%%%%%%%%%%%%%%%%%%%%%%%%%%%%%%%%%%%%%%%%%%%%%%%%%%%
%
%

\subsubsection{The effect of dimension $n$}
\label{sec:comp-results-dimension}

 Finally, we analyze the effect of dimension~$n$  on the results, which are reported 
in Figure~\ref{fig:gapTime-n}. The most interesting observation here is that the average cumulative gap of the \textbf{P} relaxation is highest when $n=3$. 
%{This behavior is somewhat consistent with Proposition~\ref{prop:dbound-stSimplex-lb}, in which we have shown that the lower bound on the worst-case distance between  the nonconvex set  $\cS^\kappa$ and its convex hull is attained for small values of  $n$.} 
We also note that the average cumulative gap gradually decreases with higher values of $n$, which is in agreement with the  simulation results that we have reported in Section~\ref{sec:objFunc-approx} (see Tables~\ref{tab:simulateUniform} and~\ref{tab:simulateNormal}) and Theorem~\ref{thm:POWexactUniform}. The other relaxations give  similar average cumulative gaps for small values of $n$  whereas their gaps are getting more different with larger values of $n$. 

\begin{figure}[H]
	\caption{Average Cumulative Gap vs. Average Cumulative Time with respect to different $n$ values.}
	\label{fig:gapTime-n}
	
	\begin{subfigure}{.34\textwidth}
		\begin{tikzpicture}[scale=0.53]
			\begin{axis}[
				title={\large{$n=2$}},
				xmin=0, 
                xtick={0,45,...,270},
				xmax=270, xtick={0,45,...,270},
				%xmode=log,
				ymin=0, 
				ymax=5,
				ylabel={\large Average Cumulative Gap},
				% label style={font=\bfseries\boldmath},
				% tick label style={font=\bfseries\boldmath},
				scatter/classes={
					NON={scale=2.5,mark=x,  color=orange, ultra thick}, 
					POW={scale=2.5,mark=+,  color=cyan, ultra thick}, 
					PR={scale=2.5,mark=diamond,  color=black}, 
					PRSOC={scale=2.5,mark=halfsquare*, color=red, mark color=white, fill=white}, 
					PRS3={scale=2.5,mark=square, color=blue}, 
					PRSDP={scale=2.5,mark=o,  color=green}, 
					PRv={scale=2.5,mark=diamond*,   color=black}, 
					PRSOCv={scale=2.5,mark=halfsquare*, color=red, mark color=red, fill=red}, 
					PRS3v={scale=2.5,mark=square*,   color=blue},
					PRSDPv={scale=2.5,mark=*,  color=green}
				},
				] 
				\addplot[
				scatter, 
				only marks,
				scatter src=explicit symbolic,
				nodes near coords*={\annotvalue},
				node near coord style={ anchor=\anchorvalue, font=\scriptsize, color=\labelcolor},
				visualization depends on={value (\thisrow{annotation}) \as \annotvalue },
				visualization depends on={  value \thisrow{anchor} \as \anchorvalue },
				]
				table[meta=label] {
					x       y            label  annotation  anchor
221.157	0	NON	1000	south east
26.772	3.520	POW	911	south east
19.839	0.601	PR	948	south east
%27.800	0.601	PRSOC	948	south east
%21.378	0.601	PRS3	948	south east
23.448	0.601	PRSDP	948	south east
30.540	0.461	PRSDPv	959	south east
27.149	0.461	PRv	959	south east
%34.766	0.461	PRSOCv	959	south east
%33.372	0.461	PRS3v	959	south east
				};
			\end{axis}
			
		\end{tikzpicture}
	\end{subfigure}
    \begin{subfigure}{.32\textwidth}
    \centering
		\begin{tikzpicture}[scale=0.53]
			\begin{axis}[
				title={\large{$n=3$}},
				xmin=0, 
                xtick={0,45,...,270},
				xmax=270, xtick={0,45,...,270},
				%xmode=log,
				ymin=0, 
				ymax=5,
			%	xlabel={\large Average Cumulative Time (s)}, 
			%	ylabel={\large Average Cumulative Gap},
                %%yticklabel=\empty,
				% label style={font=\bfseries\boldmath},
				% tick label style={font=\bfseries\boldmath},
				scatter/classes={
					NON={scale=2.5,mark=x,  color=orange, ultra thick}, 
					POW={scale=2.5,mark=+,  color=cyan, ultra thick}, 
					PR={scale=2.5,mark=diamond,  color=black}, 
					PRSOC={scale=2.5,mark=halfsquare*, color=red, mark color=white, fill=white}, 
					PRS3={scale=2.5,mark=square, color=blue}, 
					PRSDP={scale=2.5,mark=o,  color=green}, 
					PRv={scale=2.5,mark=diamond*,   color=black}, 
					PRSOCv={scale=2.5,mark=halfsquare*, color=red, mark color=red, fill=red}, 
					PRS3v={scale=2.5,mark=square*,   color=blue},
					PRSDPv={scale=2.5,mark=*,  color=green}
				},
				] 
				\addplot[
				scatter, 
				only marks,
				scatter src=explicit symbolic,
				nodes near coords*={\annotvalue},
				node near coord style={ anchor=\anchorvalue, font=\scriptsize, color=\labelcolor},
				visualization depends on={value (\thisrow{annotation}) \as \annotvalue },
				visualization depends on={  value \thisrow{anchor} \as \anchorvalue },
				]
				table[meta=label] {
					x       y            label  annotation  anchor
227.087	0	NON	1000	south east
29.842	4.687	POW	878	north east
21.892	1.175	PR	922	south east
%40.400	0.874	PRSOC	931	south east
%41.689	0.859	PRS3	932	south east
25.807	0.859	PRSDP	932	south east
46.148	0.565	PRSDPv	951	south east
41.418	0.572	PRv	950	south east
%58.223	0.568	PRSOCv	951	south east
%60.172	0.565	PRS3v	951	south east
				};
			\end{axis}
			
		\end{tikzpicture}
	\end{subfigure}
    \hfill
    	\begin{subfigure}{.32\textwidth}
    \centering
		\begin{tikzpicture}[scale=0.53]
			\begin{axis}[
				title={\large{$n=4$}},
				xmin=0, 
                xtick={0,45,...,270},
				xmax=270, xtick={0,45,...,270},
				%xmode=log,
				ymin=0, 
				ymax=5,
				%xlabel={\large Average Cumulative Time (s)}, 
				%ylabel={\large Average Cumulative Gap},
                %%yticklabel=\empty,
				% label style={font=\bfseries\boldmath},
				% tick label style={font=\bfseries\boldmath},
				scatter/classes={
					NON={scale=2.5,mark=x,  color=orange, ultra thick}, 
					POW={scale=2.5,mark=+,  color=cyan, ultra thick}, 
					PR={scale=2.5,mark=diamond,  color=black}, 
					PRSOC={scale=2.5,mark=halfsquare*, color=red, mark color=white, fill=white}, 
					PRS3={scale=2.5,mark=square, color=blue}, 
					PRSDP={scale=2.5,mark=o,  color=green}, 
					PRv={scale=2.5,mark=diamond*,   color=black}, 
					PRSOCv={scale=2.5,mark=halfsquare*, color=red, mark color=red, fill=red}, 
					PRS3v={scale=2.5,mark=square*,   color=blue},
					PRSDPv={scale=2.5,mark=*,  color=green}
				},
				] 
				\addplot[
				scatter, 
				only marks,
				scatter src=explicit symbolic,
				nodes near coords*={\annotvalue},
				node near coord style={ anchor=\anchorvalue, font=\scriptsize, color=\labelcolor},
				visualization depends on={value (\thisrow{annotation}) \as \annotvalue },
				visualization depends on={  value \thisrow{anchor} \as \anchorvalue },
				]
				table[meta=label] {
					x       y            label  annotation  anchor
232.710	0	NON	1000	south east
33.620	4.156	POW	868	south east
26.387	1.336	PR	910	south east
%55.797	0.838	PRSOC	925	south east
%60.974	0.820	PRS3	928	south east
32.255	0.820	PRSDP	928	south east
64.361	0.488	PRSDPv	950	south east
59.050	0.495	PRv	948	south east
%88.447	0.491	PRSOCv	949	south east
%99.055	0.488	PRS3v	950	south east
				};
			\end{axis}
			
		\end{tikzpicture}
	\end{subfigure}

		\begin{subfigure}{.34\textwidth}
		\begin{tikzpicture}[scale=0.53]
			\begin{axis}[
				title={\large{$n=5$}},
				xmin=0, 
				xmax=270, xtick={0,45,...,270},
				ymin=0, 
				ymax=5,
               % %yticklabel=\empty,
				%    xlabel={\%Gap (delta)}, 
			ylabel={\large Average Cumulative Gap},
				% label style={font=\bfseries\boldmath},
				% tick label style={font=\bfseries\boldmath},
				scatter/classes={
					NON={scale=2.5,mark=x,  color=orange, ultra thick}, 
					POW={scale=2.5,mark=+,  color=cyan, ultra thick}, 
					PR={scale=2.5,mark=diamond,  color=black}, 
					PRSOC={scale=2.5,mark=halfsquare*, color=red, mark color=white, fill=white}, 
					PRS3={scale=2.5,mark=square, color=blue}, 
					PRSDP={scale=2.5,mark=o,  color=green}, 
					PRv={scale=2.5,mark=diamond*,   color=black}, 
					PRSOCv={scale=2.5,mark=halfsquare*, color=red, mark color=red, fill=red}, 
					PRS3v={scale=2.5,mark=square*,   color=blue},
					PRSDPv={scale=2.5,mark=*,  color=green}
				},
				] 
				\addplot[
				scatter, 
				only marks,
				scatter src=explicit symbolic,
				nodes near coords*={\annotvalue},
				node near coord style={ anchor=\anchorvalue, font=\scriptsize, color=\labelcolor},
				visualization depends on={value (\thisrow{annotation}) \as \annotvalue },
				visualization depends on={  value \thisrow{anchor} \as \anchorvalue },
				]
				table[meta=label] {
					x       y            label  annotation  anchor
237.948	0	NON	1000	south east
35.346	4.143	POW	873	south east
33.913	1.418	PR	912	south east
%78.322	0.824	PRSOC	929	south east
%92.027	0.796	PRS3	933	south east
36.139	0.796	PRSDP	933	south east
90.187	0.474	PRSDPv	953	south east
79.620	0.482	PRv	952	south east
%120.313	0.477	PRSOCv	953	south east
%158.413	0.474	PRS3v	953	south east
				};
			\end{axis}
			
		\end{tikzpicture}
	\end{subfigure}
		\begin{subfigure}{.32\textwidth}
            \centering
		\begin{tikzpicture}[scale=0.53]
			\begin{axis}[
				title={\large{$n=6$}},
				xmin=0, 
				xmax=270, xtick={0,45,...,270},
				%xmode=log,
				ymin=0, 
				ymax=5,
			%	xlabel={\large Average Cumulative Time (s)}, 
				%ylabel={\large Average Cumulative Gap},
                %yticklabel=\empty,
				% label style={font=\bfseries\boldmath},
				% tick label style={font=\bfseries\boldmath},
				scatter/classes={
					NON={scale=2.5,mark=x,  color=orange, ultra thick}, 
					POW={scale=2.5,mark=+,  color=cyan, ultra thick}, 
					PR={scale=2.5,mark=diamond,  color=black}, 
					PRSOC={scale=2.5,mark=halfsquare*, color=red, mark color=white, fill=white}, 
					PRS3={scale=2.5,mark=square, color=blue}, 
					PRSDP={scale=2.5,mark=o,  color=green}, 
					PRv={scale=2.5,mark=diamond*,   color=black}, 
					PRSOCv={scale=2.5,mark=halfsquare*, color=red, mark color=red, fill=red}, 
					PRS3v={scale=2.5,mark=square*,   color=blue},
					PRSDPv={scale=2.5,mark=*,  color=green}
				},
				] 
				\addplot[
				scatter, 
				only marks,
				scatter src=explicit symbolic,
				nodes near coords*={\annotvalue},
				node near coord style={ anchor=\anchorvalue, font=\scriptsize, color=\labelcolor},
				visualization depends on={value (\thisrow{annotation}) \as \annotvalue },
				visualization depends on={  value \thisrow{anchor} \as \anchorvalue },
				]
				table[meta=label] {
					x       y            label  annotation  anchor
239.787	0	NON	1000	south east
34.870	3.824	POW	880	south east
38.347	1.339	PR	918	south east
%94.191	0.694	PRSOC	936	south east
%136.466	0.659	PRS3	939	south east
42.600	0.658	PRSDP	939	south east
122.508	0.357	PRSDPv	961	south east
104.968	0.370	PRv	959	south east
%159.394	0.361	PRSOCv	960	south east
%238.788	0.357	PRS3v	961	south east
				};
			\end{axis}
			
		\end{tikzpicture}
	\end{subfigure}
    \hfill
\begin{subfigure}{.32\textwidth}
            \centering
		\begin{tikzpicture}[scale=0.53]
			\begin{axis}[
				title={\large{$n=7$}},
				xmin=0, 
				xmax=270, xtick={0,45,...,270},
				%xmode=log,
				ymin=0, 
				ymax=5,
				%xlabel={\large Average Cumulative Time (s)}, 
				%ylabel={\large Average Cumulative Gap},
                %yticklabel=\empty,
				% label style={font=\bfseries\boldmath},
				% tick label style={font=\bfseries\boldmath},
				scatter/classes={
					NON={scale=2.5,mark=x,  color=orange, ultra thick}, 
					POW={scale=2.5,mark=+,  color=cyan, ultra thick}, 
					PR={scale=2.5,mark=diamond,  color=black}, 
					PRSOC={scale=2.5,mark=halfsquare*, color=red, mark color=white, fill=white}, 
					PRS3={scale=2.5,mark=square, color=blue}, 
					PRSDP={scale=2.5,mark=o,  color=green}, 
					PRv={scale=2.5,mark=diamond*,   color=black}, 
					PRSOCv={scale=2.5,mark=halfsquare*, color=red, mark color=red, fill=red}, 
					PRS3v={scale=2.5,mark=square*,   color=blue},
					PRSDPv={scale=2.5,mark=*,  color=green}
				},
				] 
				\addplot[
				scatter, 
				only marks,
				scatter src=explicit symbolic,
				nodes near coords*={\annotvalue},
				node near coord style={ anchor=\anchorvalue, font=\scriptsize, color=\labelcolor},
				visualization depends on={value (\thisrow{annotation}) \as \annotvalue },
				visualization depends on={  value \thisrow{anchor} \as \anchorvalue },
				]
				table[meta=label] {
					x       y            label  annotation  anchor
242.778	0	NON	1000	south east
35.297	3.537	POW	887	south east
46.573	1.283	PR	920	south east
%120.172	0.687	PRSOC	937	south east
%190.440	0.653	PRS3	940	south east
50.719	0.653	PRSDP	940	south east
142.417	0.352	PRSDPv	963	south east
133.004	0.358	PRv	962	south east
%220.323	0.354	PRSOCv	962	south east
%350.255	0.352	PRS3v	963	south east
				};
			\end{axis}
			
		\end{tikzpicture}
	\end{subfigure}

    	\begin{subfigure}{.34\textwidth}
		\begin{tikzpicture}[scale=0.53]
			\begin{axis}[
				title={\large{$n=8$}},
				xmin=0, 
                xtick={0,45,...,270},
				xmax=270, xtick={0,45,...,270},
				%xmode=log,
				ymin=0, 
				ymax=5,
			    xlabel={\large Average Cumulative Time (s)}, 
			ylabel={\large Average Cumulative Gap},
				% label style={font=\bfseries\boldmath},
				% tick label style={font=\bfseries\boldmath},
				scatter/classes={
					NON={scale=2.5,mark=x,  color=orange, ultra thick}, 
					POW={scale=2.5,mark=+,  color=cyan, ultra thick}, 
					PR={scale=2.5,mark=diamond,  color=black}, 
					PRSOC={scale=2.5,mark=halfsquare*, color=red, mark color=white, fill=white}, 
					PRS3={scale=2.5,mark=square, color=blue}, 
					PRSDP={scale=2.5,mark=o,  color=green}, 
					PRv={scale=2.5,mark=diamond*,   color=black}, 
					PRSOCv={scale=2.5,mark=halfsquare*, color=red, mark color=red, fill=red}, 
					PRS3v={scale=2.5,mark=square*,   color=blue},
					PRSDPv={scale=2.5,mark=*,  color=green}
				},
				] 
				\addplot[
				scatter, 
				only marks,
				scatter src=explicit symbolic,
				nodes near coords*={\annotvalue},
				node near coord style={ anchor=\anchorvalue, font=\scriptsize, color=\labelcolor},
				visualization depends on={value (\thisrow{annotation}) \as \annotvalue },
				visualization depends on={  value \thisrow{anchor} \as \anchorvalue },
				]
				table[meta=label] {
					x       y            label  annotation  anchor
243.463	0	NON	1000	south east
35.458	3.103	POW	893	south east
49.130	1.223	PR	925	south east
%151.174	0.697	PRSOC	941	south east
%254.877	0.661	PRS3	944	south east
59.067	0.660	PRSDP	944	south east
181.262	0.357	PRSDPv	963	south east
168.408	0.363	PRv	961	south east
%269.813	0.359	PRSOCv	962	south east
%480.194	0.357	PRS3v	963	south east
				};
			\end{axis}
			
		\end{tikzpicture}
	\end{subfigure}
	%%%%%%%%%%%%%%%%%%%%%%%%
	\begin{subfigure}{.32\textwidth}
        \centering
		\begin{tikzpicture}[scale=0.53]
			\begin{axis}[
				title={\large{$n=9$}},
				xmin=0, 
                xtick={0,45,...,270},
				xmax=270, xtick={0,45,...,270},
				%xmode=log,
				ymin=0, 
				ymax=5,
				xlabel={\large Average Cumulative Time (s)}, 
				%ylabel={\large Average Cumulative Gap},
                %yticklabel=\empty,
				% label style={font=\bfseries\boldmath},
				% tick label style={font=\bfseries\boldmath},
				scatter/classes={
					NON={scale=2.5,mark=x,  color=orange, ultra thick}, 
					POW={scale=2.5,mark=+,  color=cyan, ultra thick}, 
					PR={scale=2.5,mark=diamond,  color=black}, 
					PRSOC={scale=2.5,mark=halfsquare*, color=red, mark color=white, fill=white}, 
					PRS3={scale=2.5,mark=square, color=blue}, 
					PRSDP={scale=2.5,mark=o,  color=green}, 
					PRv={scale=2.5,mark=diamond*,   color=black}, 
					PRSOCv={scale=2.5,mark=halfsquare*, color=red, mark color=red, fill=red}, 
					PRS3v={scale=2.5,mark=square*,   color=blue},
					PRSDPv={scale=2.5,mark=*,  color=green}
				},
        %   legend entries={NON, POW, PR, PRSOC, PRS3, PRSDP, PRv, PRSOCv, PRS3v, PRSDPv},
       %             legend style={
        %at={(1.03,1.15)}, anchor=north,
       % legend columns=2},
				] 
				\addplot[
				scatter, 
				only marks,
				scatter src=explicit symbolic,
				nodes near coords*={\annotvalue},
				node near coord style={ anchor=\anchorvalue, font=\scriptsize, color=\labelcolor},
				visualization depends on={value (\thisrow{annotation}) \as \annotvalue },
				visualization depends on={  value \thisrow{anchor} \as \anchorvalue },
				]
				table[meta=label] {
					x       y            label  annotation  anchor
245.389	0	NON	1000	south east
37.425	2.838	POW	897	south east
52.451	1.062	PR	930	south east
%179.946	0.551	PRSOC	949	south east
%359.351	0.519	PRS3	953	south east
61.423	0.517	PRSDP	953	south east
220.986	0.296	PRSDPv	971	south east
209.418	0.302	PRv	969	south east
%321.950	0.299	PRSOCv	970	south east
%665.176	0.296	PRS3v	971	south east
				};
			\end{axis}
			
		\end{tikzpicture}
	\end{subfigure}
    %%%%%%%%%%%%%%%%%%%%%%%%	%
%
\hfill
	\begin{subfigure}{.32\textwidth}
        \centering
		\begin{tikzpicture}[scale=0.53]
			\begin{axis}[
				title={\large{$n=10$}},
				xmin=0, 
				xmax=270, xtick={0,45,...,270},
				%xmode=log,
				ymin=0, 
				ymax=5,
				xlabel={\large Average Cumulative Time (s)}, 
				%ylabel={\large Average Cumulative Gap},
                %yticklabel=\empty,
				% label style={font=\bfseries\boldmath},
				% tick label style={font=\bfseries\boldmath},
				scatter/classes={
					NON={scale=2.5,mark=x,  color=orange, ultra thick}, 
					POW={scale=2.5,mark=+,  color=cyan, ultra thick}, 
					PR={scale=2.5,mark=diamond,  color=black}, 
				%	PRSOC={scale=2.5,mark=halfsquare*, color=red, mark color=white, fill=white}, 
				%	PRS3={scale=2.5,mark=square, color=blue}, 
					PRSDP={scale=2.5,mark=o,  color=green}, 
					PRv={scale=2.5,mark=diamond*,   color=black}, 
				%	PRSOCv={scale=2.5,mark=halfsquare*, color=red, mark color=red, fill=red}, 
				%	PRS3v={scale=2.5,mark=square*,   color=blue},
					PRSDPv={scale=2.5,mark=*,  color=green}
				},
           legend entries={\textbf{NON}, \textbf{P}, \textbf{PR},  \textbf{PRS}, \textbf{PRV},   \textbf{PRSV}},
                    legend style={
        legend columns=2},
				] 
				\addplot[
				scatter, 
				only marks,
				scatter src=explicit symbolic,
				nodes near coords*={\annotvalue},
				node near coord style={ anchor=\anchorvalue, font=\scriptsize, color=\labelcolor},
				visualization depends on={value (\thisrow{annotation}) \as \annotvalue },
				visualization depends on={  value \thisrow{anchor} \as \anchorvalue },
				]
				table[meta=label] {
					x       y            label  annotation  anchor
246.544	0	NON	1000	south east
41.230	2.658	POW	902	south east
60.226	1.008	PR	931	south east
%211.585	0.505	PRSOC	948	south east
%460.328	0.484	PRS3	950	south east
70.285	0.484	PRSDP	950	south east
267.023	0.235	PRSDPv	970	south east
267.560	0.241	PRv	969	south east
%399.119	0.238	PRSOCv	969	south east
%874.760	0.235	PRS3v	970	south east
				};
			\end{axis}
			
		\end{tikzpicture}
	\end{subfigure}
    
    \end{figure}
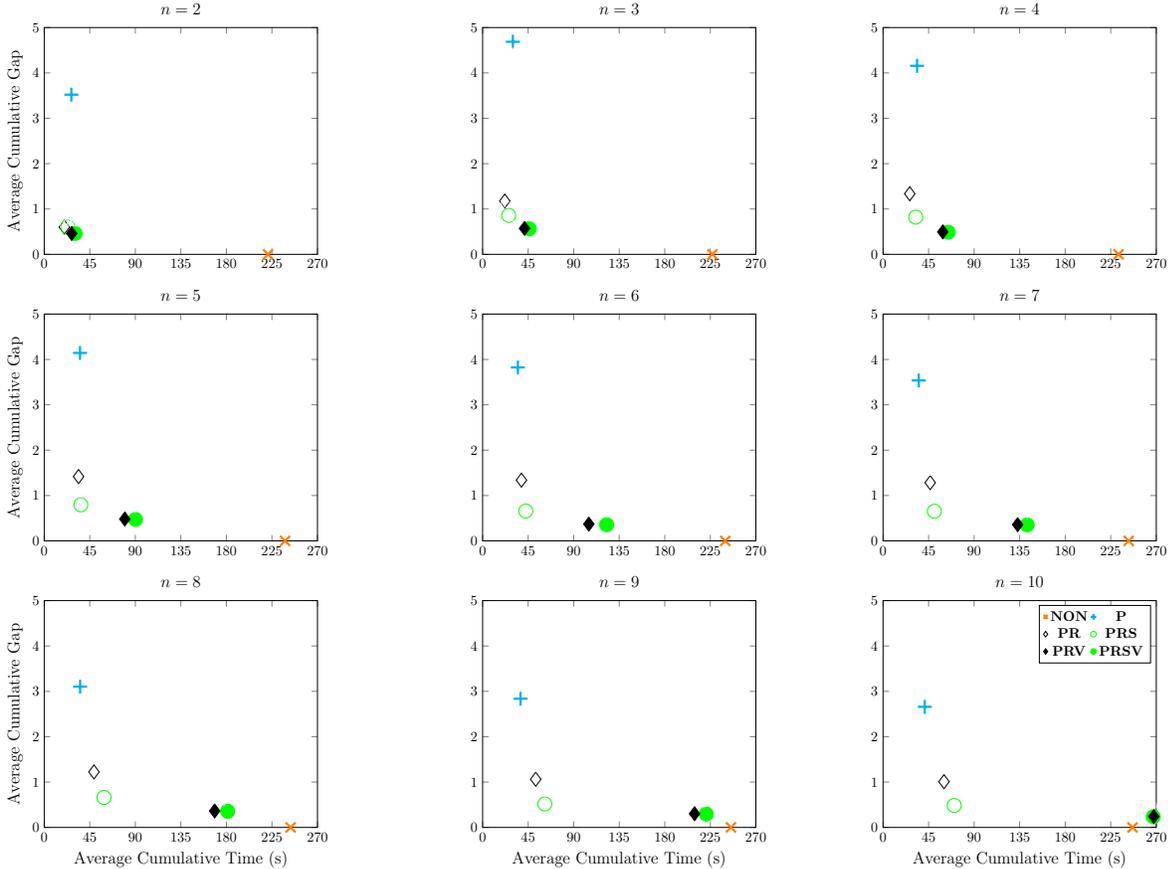
In terms of computational cost, we observe a slight increase in the   solution of the \textbf{NON} model  with $n$ whereas the times of \textbf{P}, \textbf{PR} and \textbf{PRS} relaxations increase moderately. We note that the times of \textbf{PRV} and \textbf{PRSV} relaxations are most sensitive to the dimension $n$ and they become even more expensive to solve than the \textbf{NON} model  for $n=10$. Combined with the fact that the effect of the RPT constraints~\eqref{eq:rpt}  diminishes with dimension, it is not  advisable to use them in higher dimensions.

\section{Conclusion}
\label{sec:conclusion}

%In this paper, we studied the set $\cS^\kappa$, which is directly related to   separable standard quadratic programming and appears as a substructure in potential-based flow networks. We proposed several conic relaxations for this nonconvex set, and obtained the exact convex hull in low-dimensions for $\kappa=2$ using the \textbf{PR} and \textbf{PRS} relaxations. We provided  justifications as to why the \textbf{P} relaxation, which is the weakest relaxation we consider, is expected to provide tight bounds when a linear function is optimized over it. In particular, we derived worst-case bounds on the distance between the \textbf{P} relaxation and the nonconvex set, and performed a probabilistic analysis which demonstrates that this relaxation is exact with high probability in higher dimensions. 
%%Beyond $\textbf{P}$, our results that RLT relaxation is very important. In particular: (i)  $\textbf{PR}$ gives significantly better bounds, and (ii) while addition of PSD or RPT constraints do not add any value on top on $\textbf{P}$ when added to $\textbf{PR}$ they produce significant improvement to bounds. Finally, $\textbf{PRS}$ and  $\textbf{PRV}$, are both improvements over $\textbf{PR}$, with $\textbf{PRS}$ faster to run, but giving slightly worse bounds than $\textbf{PRV}$. 
%We also carried out an extensive set of experiments, comparing the nine relaxations we proposed.

%\textcolor{blue}
{In this paper, we studied the nonconvex set $\cS^\kappa$, which is directly related to   separable standard quadratic programming and appears as a substructure in potential-based flow networks. We proposed several conic relaxations for this nonconvex set, and compared these relaxations both theoretically  and via extensive computational  experiments. Our main conclusions are as follows. The $\textbf{P}$ relaxation, which is the weakest relaxation among those considered, provides quite strong bounds when a linear function is optimized over it. In particular, we derived worst-case bounds on the distance between the $\textbf{P}$ relaxation and $\cS^\kappa$, and performed a probabilistic analysis which demonstrates that this relaxation is exact with high probability in higher dimensions. Beyond $\textbf{P}$, our results highlight that the RLT relaxation is very important. In particular, the exact convex hull in low-dimensions for $\kappa=2$ is obtained using the $\textbf{PR}$ relaxation. Computationally, we see that $\textbf{PR}$ gives significantly better bounds in comparison to $\textbf{P}$. Interestingly, we proved that addition of PSD constraints do not add any value on top on $\textbf{P}$, while when these are added to $\textbf{PR}$ can produce non-trivial improvement to bounds. Finally, $\textbf{PRS}$ and  $\textbf{PRV}$ are incomparable to each other, both produce improvements over $\textbf{PR}$, with $\textbf{PRS}$ being more efficient to solve, but giving slightly worse bounds than $\textbf{PRV}$ on average. 
}

We would like to note that although the separable function $\sum_{j=1}^n (\alpha_j x_j + \beta_j x_j^\kappa)$ we consider in this study may seem special at first glance, most of our results extend to the more general case where we replace each    function  $\alpha_j x_j + \beta_j x_j^\kappa$ with a general convex function of $x_j$ that is non-decreasing on the ground set. 

There are several future research avenues we would like to explore. For instance, resolving Conjecture~\ref{conj:exactness-kappa=2-stSimplexn-any}, which states that the \textbf{PRS} relaxation is exact for $\kappa=2$, and Conjecture~\ref{conj:PRS=PRs3}, which states that  \textbf{PRS${^3}$} and \textbf{PRS} relaxations are equivalent,  are two immediate directions for further theoretical analysis. Moreover, analyzing the cases with a separable function with multiple exponents, exponents where $\kappa<1$ and   different ground sets from both theoretical and empirical aspects are also promising.  

% {\color{red}some\\ conclusion\\ to\\ be\\ added...

% repeat the conjecture, talk about possible extensions, different exponents, specific cases etc.

% future: more experiments with different ground sets

% %almost all results work for general convex increasing function  (not only $x^\kappa$)
% }

\section*{Acknowledgments}
We gratefully acknowledge several helpful discussions with Nick Sahinidis on this topic.

\bibliographystyle{plain}
\bibliography{references}

\end{document}